\documentclass[11pt,twoside]{amsart}
\usepackage{fancyhdr}

\usepackage{amsmath}
\usepackage{amssymb, amsthm}
\usepackage{graphicx}
\usepackage[colorlinks,linkcolor=blue,citecolor=blue]{hyperref}
\usepackage{tikz-cd}

\numberwithin{equation}{section}
\newtheorem{thm}{Theorem}[section]
\newtheorem{lem}[thm]{Lemma}

\newtheorem{cor}[thm]{Corollary}
\newtheorem{pro}[thm]{Proposition}
\newtheorem{rmk}[thm]{Remark}

\newtheorem{ex}[thm]{Example}

\newcommand{\bess}{\begin{eqnarray*}}
	\newcommand{\eess}{\end{eqnarray*}}
\newcommand{\C}{\mathbb{\widehat{C}}}

\renewcommand{\setminus}{-}
\renewcommand{\epsilon}{\varepsilon}

\input xy
\xyoption{all}

\title{Boundary of the central hyperbolic component II:  boundary extension theorem}

  \author{Jie Cao}
\address{(J. Cao) School of Mathematical Sciences, Shenzhen University, Shenzhen, 518060, China}
\email{mathcj@foxmail.com}

\author{Xiaoguang Wang}
\address{(X. Wang) School of Mathematical Sciences, Zhejiang University, Hangzhou, 310027, China}
\email{wxg688@163.com}

\author{Yongcheng Yin}
\address{(Y. Yin) School of Mathematical Sciences, Zhejiang University, Hangzhou, 310027, China}
\email{yin@zju.edu.cn}

\begin{document}

  \begin{abstract}  In this paper, we study the boundary behavior of   Milnor's parameterization $\Phi: \mathcal B_d\rightarrow \mathcal H_d$ of the central hyperbolic component $\mathcal H_d$ via Blaschke products.
  	We  establish a  boundary extension theorem   by giving a necessary and sufficient condition for $D\in \partial \mathcal B_d$ which allows $\Phi$-extension.   Further we show that cusps are  dense in a  full Hausdorff dimensional subset of $\partial \mathcal H_d$,  
  	partially confirming a conjecture of McMullen.
\end{abstract}

\subjclass[2020]{Primary 37F45; Secondary 37F10, 37F15}


\keywords{central hyperbolic component, boundary extension, local connectivity, cusps}

  \date{\today} 
  
  	\maketitle


	
	\section{Introduction}

Let $\mathcal P_d$ be the space  of degree $d\geq 3$ monic  polynomials 
$$f(z)=a_1z+\cdots+a_{d-1} z^{d-1}+z^d,$$
where $(a_1, \cdots, a_{d-1})\in \mathbb C^{d-1}$. A polynomial $f\in \mathcal P_d$ is {\it hyperbolic} if the orbit of each critical point   tends to  $\infty$ or a bounded attracting cycle. The collection of all hyperbolic polynomials is an open subset of $\mathcal P_d\cong  \mathbb C^{d-1}$, and each component is called a {\it hyperbolic component}.  The hyperbolic component $\mathcal H_d$ containing  $z^d$ is called  the  {\it central hyperbolic component} (or  {\it  principal hyperbolic domain},  {\it  main hyperbolic component} in literature). 

Among all hyperbolic components, the central hyperbolic component $\mathcal H_d$ is of fundamental importance in holomorphic dynamics. While the maps in $\mathcal H_d$ have  the simplest  dynamical behavior,  their bifurcations on the boundary $\partial \mathcal H_d$ exhibit  abundant variety. Viewing each map $f\in \mathcal H_d$ as the `mating' of   $z^d$ and a Blaschke  product, McMullen \cite{Mc94b} discovers analogies between $\partial\mathcal H_d$ and the geometric boundary of the Teichm\"uller space. Problems and conjectures on $\partial \mathcal H_d$ are posed in \cite{Mc94b}.  Besides these  analogies, understanding $\partial \mathcal H_d$  is a  foundational 
step to understand the   boundaries of other hyperbolic components as well as the bifurcation locus.
 
 It is known from Milnor \cite{Mil12} that $\mathcal H_d$ is a topological cell.
 DeMarco \cite{De01} shows that $\mathcal H_d$ is a  domain of holomorphy.
 Petersen and Tan \cite{PT09} construct  an analytic coordinate for $\mathcal H_3$ which can extend to a large part of $\partial \mathcal H_3$. Blokh, Oversteegen,  Ptacek and  Timorin \cite{BOPT14,BOPT16,BOPT18} give a  combinatorial model for $\partial \mathcal H_d$ and study the properties of maps in $\partial \mathcal H_d$. In \cite{Luo24}, Luo  classifies the geometrically finite polynomials on $\partial \mathcal H_d$. Recently, Gao, X. Wang and Y. Wang \cite{GWW25} prove that the locally connected part of $\partial \mathcal H_d$ has full Hausdorff dimension $2d-2$.


An effective  way to understand $\mathcal H_d$ is through its Blaschke model.
Let $\mathcal B_d$ be the space of Blaschke products of degree $d$, with $0,1$ as fixed points. 
Each $B\in \mathcal B_d$ takes the form 
$$B(z)=z\Bigg(\prod_{k=1}^{d-1}\frac{z-a_k}{1-\overline{a_k}z}\Bigg) \Bigg(\prod_{k=1}^{d-1}\frac{1-\overline{a_k}}{1- {a_k}}\Bigg), \ a_1, \cdots, a_{d-1}\in \mathbb D.$$
Milnor \cite{Mil12} shows that there is a natural homeomorphism $\Psi:\mathcal H_d\rightarrow \mathcal B_d$, defined as follows: for each $f\in \mathcal H_d$, since the Fatou set $U_{f}(0)$ containing $0$ is  a Jordan disk,
there is a unique Riemann mapping $\psi_f:  U_{f}(0)\rightarrow \mathbb D$ normalized as $\psi_f(0)=0$ and $\psi_f(\nu_f)=1$, where  $\nu_f\in \partial U_{f}(0)$  is the landing point of the $0$-external ray.  
The map $\Psi$ is defined  as $\Psi(f)=\psi_f\circ f\circ \psi_f^{-1}$. 

The  homeomorphism  $\Psi:\mathcal H_d\rightarrow \mathcal B_d$ offers a promising strategy for exploring the structure of $\partial \mathcal H_d$. 
That is,  to understand $\partial \mathcal H_d$, one needs to study the boundary behavior of  
$$\Phi=\Psi^{-1}: \mathcal B_d\rightarrow \mathcal H_d.$$

  
   The space $\mathcal B_d$ can be identified as the set ${\rm Div}_{d-1}(\mathbb D)$ of integral divisors of degree $d-1$ over the unit disk $\mathbb D$. Following McMullen \cite[\S 3]{Mc09a}, there is an algebraic compactification of   ${\mathcal B_d}$, by identifying each 
  $D=(B, S)\in \partial  {\mathcal B_d}$ as the pair of a Blaschke product $B$ of degree $1\leq l<d$ and a source of integral  divisor $S\in {\rm Div}_{d-l}(\partial \mathbb D)$.
  The boundary therefore  can be decomposed as
  $$ \partial  {\mathcal B_d}=\bigsqcup_{1\leq l< d} \Big({\mathcal B_l}\times {\rm Div}_{d-l}(\partial \mathbb D)\Big).$$


Based on the compactification $\overline{\mathcal B_d}$, the following  problem naturally arises:


\vspace{5pt}
\textbf{Boundary Extension Problem: }{\it Given any $D\in  \partial \mathcal B_d$, can $\Phi: \mathcal B_d\rightarrow \mathcal H_d$ extend continuously to $D$? }

\vspace{5pt}
 
 

Our first main result gives a complete answer to this problem.
 
 	\begin{thm}  \label{bet-main} The homeomorphism  $\Phi: {\mathcal B_d}\rightarrow \mathcal H_d$ extends continuously to $D=(B, S)\in \partial  {\mathcal B_d}$ if and only if $D$ is one 
 		of the following two types: 
 		
 	(\textbf{R}).  $D$ is regular, $S$ is simple, $1\notin {\rm supp}(S)$  and $D$ has no dynamical relation;
 	
 	(\textbf{S}).    $D$ is singular,  $S$ is simple and $1\notin {\rm supp}(S)$.
 		
 	Further,	let $\mathcal R$ and $\mathcal S$ be the sets of all $D\in \partial  {\mathcal B_d}$ of type (\textbf{R}) and	type (\textbf{S}), respectively.   Then the extension $\overline{\Phi}: {\mathcal B_d}\sqcup \mathcal R \sqcup \mathcal S\rightarrow \overline{ \mathcal H_d}$
 	 satisfies that 
 	 
 	 $\overline{\Phi}|_{\mathcal R}:\mathcal R\rightarrow \overline{\Phi}(\mathcal R)$
is  a homeomorphism,   and

 $\overline{\Phi}|_{\mathcal S}$ is the constant map   $\overline{\Phi}|_{\mathcal S}\equiv f_*$, where   $f_*(z)=z+z^d$.
  \end{thm}

See \S \ref{Blaschke}, \ref{exploration}, \ref{regular} for the basic notions (i.e. regular, singular, simple, dynamical relation) of the divisor $D$.
We remark that the subsets $\mathcal R, \mathcal S \subset  \partial{\mathcal B_d}$  have real dimensions $2d-3$ and $d-1$ respectively, while $\overline{\Phi}(\mathcal R)\subset \partial \mathcal H_d$ has maximal  Hausdorff dimension $2d-2$  \cite{GWW25}\footnote{In \cite{GWW25}, it is shown that $\Phi$ extends to  a smaller subset $\mathcal A$ of $\mathcal R$, the set of  {\it $\mathcal H$-admissible divisors}.  All maps in  $\overline{\Phi}(\mathcal A)$ are Misiurewicz, and  $\overline{\Phi}(\mathcal A)$  has  Hausdorff dimension $2d-2$.}.  
Therefore the homeomorphism  $\overline{\Phi}|_{\mathcal R}:\mathcal R\rightarrow \overline{\Phi}(\mathcal R)$  exhibits distorted behavior.

The image set  $\overline{\Phi}(\mathcal R)$ contains an abundance of maps with parabolic  cycles and accumulates at such maps in $\partial \mathcal H_d\setminus \overline{\Phi}(\mathcal R)$, therefore $\overline{\Phi}(\mathcal R)$ is  grossly distorted due to the parabolic implosion. However, as the target of a continuous extension, one might expect that $\partial \mathcal H_d$ has a nice topology near   $\overline{\Phi}(\mathcal R)$.
Our next theorem shows that this is indeed the case.

	\begin{thm}\label{lc-R}  For any $f\in \overline{\Phi}(\mathcal R)$,  $\partial \mathcal H_d$ is locally connected at $f$.
\end{thm}

Here   a  set $X$ is {\it locally connected} at $x\in X$, if there exists a family $\{U_k\}_{k\geq 1}$ of open and connected neighborhoods of $x$ in  $X$ such that $\lim_k{\rm diam}(U_k)=0$. 

    \vspace{5pt}
\noindent{\bf Remark 1.1}. \emph{A map $f\in \overline{\Phi}(\mathcal R)$ can have parabolic cycles or recurrent critical points, or both.
	  Theorem \ref{lc-R} does not mean that $\partial \mathcal H_d$ has bad topology near $\partial \mathcal H_d\setminus \overline{\Phi}(\mathcal R)$. In fact, it is conjectured that  $\partial \mathcal H_d$ is also locally connected at most maps $f\in \partial \mathcal H_d\setminus \overline{\Phi}(\mathcal R)$.} 
\vspace{5pt}


According to Luo \cite{Luo24},  when $d\geq 4$,  self-bumps occur   on $\partial \mathcal H_d$ and $\overline{\mathcal H_d}$ is not a topological manifold with boundary. This  phenomenon  means that there are different accesses approaching some map on $\partial \mathcal H_d$.  In our work, for any  sequence $(f_n)_n$ in $\mathcal H_d$ approaching some $f\in \partial\mathcal H_d$, we use  all possible aglebraic limits of    $(\Psi(f_n))_n$  to encode different ways approaching $f$.  As a by-product of the proof of  Theorem \ref{bet-main},  we show that the ways of approaching  $f_*$ can realize  all singular  divisors.
 Precisely, 
 
  \vspace{5pt}
 \noindent{\bf Corollary 1.1}(Maximal self-bumps). \emph{For any singular divisor $D=(B, S)\in \partial \mathcal B_d$,  there is a sequence $(f_n)_n$ in $\mathcal H_d$  converging to $f_*$, for which }
 	$$\Psi(f_n)\rightarrow D   \ \text{ algebraically}.$$
 

 	  
 	  The Boundary Extension Theorem (Theorem \ref{bet-main}) demonstrates its efficacy in elucidating the structure and fundamental 
 	   properties of $\partial \mathcal H_d$.  Specifically, it allows us to study the distribution of cusps in $\partial \mathcal H_d$.
 	  Recall that a rational map $f$ is  {\it geometrically finite} if the critical points in the Julia set $J(f)$ have finite orbits.  A {\it cusp} is  a geometrically finite map with parabolic cycles. Based on his celebrated work \cite{Mc91} and the analogies between rational maps and   Teichm\"uller theory,   McMullen posed the following 
 	  
 	    \vspace{5pt}
 	  \noindent{\bf Conjecture 1.1} (\cite{Mc94b}). \emph{Cusps are dense in $\partial \mathcal H_d$.} 
 	  \vspace{5pt}
 	  
   	
   	Faught \cite{F92} and Roesch \cite{R}  have shown that the boundary $\partial \mathcal H_d$, when considered within the one parameter family $f_a(z)=az^{d-1}+z^d$ where $a\in \mathbb C$,  is a Jordan curve, which provides an  evidence  of Conjecture 1.1 in a slice.

 Our last main theorem shows that cusps are dense in the   full Hausdorff dimensional  subset $\overline{\Phi}(\mathcal R)$ of $\partial \mathcal H_d$, 
   partially confirming this conjecture.

 	\begin{thm} \label{cusp-dense} Cusps are dense in $\overline{\Phi}(\mathcal R)$. Precisely, for any $f\in \overline{\Phi}(\mathcal R)$, any $\epsilon>0$, any 
 		integers $m, n\geq 0$ satisfying that 
 		$$m\geq 1, \  m+n\leq d-{\rm deg}(f|_{U_f(0)}),$$
 		there is a geometrically finite polynomial  $g\in \overline{\Phi}(\mathcal R) \cap  \mathcal N_\epsilon(f)$, which has exactly $m$ parabolic cycles and $n$ critical points on $\partial U_g(0)$.
 	\end{thm}
 

 The paper is organized as follows: 
 
In \S  \ref{Blaschke},  we prove   a boundary extension theorem (Theorem \ref{bet}) for the parameterization of  Blaschke products via critical points.  In  \S \ref{continuity-property}, some continuity properties (for pointed disks, rays, maps) are established.  In \S  \ref{exploration},   each $D\in \partial \mathcal B_d$ is associated with a connected compact set $I_{\Phi}(D)\subset \partial \mathcal H_d$,  consisting of all possible limits of $(\Phi(B_n))_n$ for the sequences $(B_n)_n$ in $\mathcal B_d$ converging to $D$.  
The  Boundary Extension Problem is then reduced to classify those $D$ for which $I_{\Phi}(D)$ is a singleton. In \S \ref{regular} and
 \S \ref{regular-nss}, we study $I_{\Phi}(D)$ for regular divisors. In \S \ref{sing-div}, we study $I_{\Phi}(D)$ for singular divisors.  In \S \ref{proof-main}, we prove Theorems \ref{bet-main},  \ref{lc-R} and  \ref{cusp-dense}.
 
 This paper   extends the work of \cite{CWY}, in which the local connectivity of Julia sets  and rigidity theorem  were established for the maps in the regular part of $\partial \mathcal H_d$. The rigidity is applied to study  the extension of $\Phi$ in   \S \ref{regular}. 
 
 	

 \vspace{5pt}
  \noindent\textbf{Acknowledgments.}   The research is supported by National Key R\&D Program of China (Grants No. 2021YFA1003200 and No. 2021YFA1003202),  National Natural Science Foundation of China (Grants No. 12131016 and No. 12331004), and  the Fundamental Research Funds for the Central Universities 2024FZZX02-01-01.

 		
 	

	\section{Blaschke products}\label{Blaschke}
	
	Throughout the paper we adopt the following notations:
	
	\begin{itemize}
		\item $\mathbb C$ and $\C$:  the complex plane and the Riemann sphere
		
		\item $\mathbb N$ and $\mathbb Z$: the set of natural numbers $0,1,2,\cdots$ and  the set of integers
		
		\item $\mathbb D(a,r)=\{z\in \mathbb C; |z-a|<r\}$, $\mathbb D=\mathbb D(0,1)$
		
		\item $d_{U}(a, b)$: the hyperbolic distance between $a,b$ in a Jordan disk $U$
		
		\item $\mathbb D_{\rm hyp}(a,r)$: the hyperbolic disk in $\mathbb D$, centered at $a$ with radius $r$ 
		
		\item ${\rm diam}(E)$: the Euclidean diameter $\sup_{a,b\in E}|a-b|$ of a  set $E\subset \mathbb C$

		\item   A  sequence of maps  $(f_n)_n$ {\it converges}  to $f$ in a domain $\Omega$ means that $f_n$  \textbf{converges locally and uniformly} to $f$ in $\Omega$.
		
	\end{itemize}

	A {\it divisor} $D$ on a set $\Omega\subset \mathbb C$ is a formal sum
	$$D=\sum_{q\in \Omega} \nu(q) \cdot q,$$
	where $\nu: \Omega\rightarrow \mathbb Z$ is a map, $\nu(q)\neq 0$ for only finitely many $q\in \Omega$.
	The {\it support} of   $D$, denoted by ${\rm supp}(D)$, is the finite set $\{q\in \Omega; \nu(q)\neq 0\}$.  
	The divisor $D$ is called {\it integral} (or {\it effective}) if $\nu\geq 0$;    {\it simple} if $\ \nu(q)=1$ for all $q\in {\rm supp}(D)$.  The {\it degree} of an integral divisor $D$ is defined by 
	${\rm deg}(D)=\sum_{q\in \Omega} \nu(q)$.
	
	Let  ${\rm Div}_e(\Omega)$ be the set of all integral divisors on $\Omega$ of degree $e\geq 1$. 
	There is a natural quotient map from $\Omega^e$ to ${\rm Div}_e(\Omega)$
	 sending an ordered $e$-tuple $(z_1, \cdots, z_e)\in \Omega^e$ to $D=\sum_{1\leq k\leq e} 1 \cdot z_k$. This implies that when $\Omega$ is a planar set, ${\rm Div}_e(\Omega)$ inherits a   quotient   topology.
	
	For any integers $e, m\geq 1$, let $\mathcal B_{e,m}$ be the space of   Blaschke product $f$ of degree $e+m$, with $f(0)=0, f(1)=1$ and   local degree \footnote{The {\it local degree} ${\rm deg}(f, q)$ is the multiplicity of $q$ as the zero of $f(z)-f(q)$.} ${\rm deg}(f, 0)\geq m$:
	$$f(z)=z^m\prod_{k=1}^e\bigg(\frac{1-\overline{a_k}}{1-{a_k}} \cdot \frac{z-a_k}{1-\overline{a_k}z}\bigg), \ a_1, \cdots, a_e\in \mathbb D.$$ 
 Clearly  each $f\in \mathcal B_{e,m}$ is uniquely determined by its {\it zero divisor}
 $$Z(f):=m\cdot 0+\sum_{k=1}^{e} 1\cdot a_k=: m\cdot 0+Z_f.$$
	The  critical set of $f$ in $\mathbb D$  induces  the {\it ramification divisor} $R(f)$,   defined by
	$$R(f)=\sum_{q\in \mathbb D}({\rm deg}(f,q)-1)\cdot q=: (m-1)\cdot 0 +R_f.$$
	We call $Z_f$  and $R_f$  the {\it free zero divisor} and  the {\it free ramification divisor}. 
	Clearly   $f\mapsto Z_f$ gives a bijection from $\mathcal B_{e,m}$ to ${\rm Div}_e(\mathbb D)$, so one can identify  $\mathcal B_{e, m}$ with ${\rm Div}_e(\mathbb D)$ by this map. Since $R(f)$ is  uniquely determined  by its free part  $R_f\in {\rm Div}_e(\mathbb D)$, there is a natural self map of ${\rm Div}_e(\mathbb D)$: $Z_f\mapsto R_f$.

	\begin{thm} [Heins \cite{H}, Zakeri \cite{Z}]\label{z-p}  For any integer $e\geq 1$, the map $$\Psi_{e,m}:
		\begin{cases} {\rm Div}_e(\mathbb D)\rightarrow {\rm Div}_e(\mathbb D)\\
			Z_f\mapsto R_f \end{cases}$$
		is a homeomorphism.
	\end{thm}
	
	Theorem \ref{z-p} implies  that each $f\in \mathcal B_{e,m}$ is uniquely determined by its free ramification divisor $R_f\in {\rm Div}_e(\mathbb D)$, and each
	  $R\in {\rm Div}_e(\mathbb D)$ can be realized as a free ramification divisor of  a unique $f\in \mathcal B_{e,m}$.
	
	The main purpose of this section is to show that the map $\Psi_{e,m}$ can   extend  to the closure 
	$\overline{{\rm Div}_e(\mathbb D)}={\rm Div}_e(\overline{\mathbb D})$, with a nice boundary behavior. For this end, it is worth noting  the set-theoretic expression
	
\begin{equation}\label{decom-div}{\rm Div}_e(\overline{\mathbb D})=\bigsqcup_{d_1+d_2=e; \ d_1,d_2\geq 0}\Big({\rm Div}_{d_1}({\mathbb D})+
	{\rm Div}_{d_2}(\partial{\mathbb D})\Big), \end{equation}
	
\begin{equation}\label{decom-boundary}\partial {\rm Div}_e(\overline{\mathbb D})=\bigsqcup_{d_1+d_2=e; \  d_2\geq 1}\Big({\rm Div}_{d_1}({\mathbb D})+
	{\rm Div}_{d_2}(\partial{\mathbb D})\Big).\end{equation}
	
		Note that  $\partial {\rm Div}_e({\mathbb D})=\partial {\rm Div}_e(\overline{\mathbb D})$ in its topology. If $D\in  \partial {\rm Div}_e({\mathbb D})$, then 
	$D\in  {\rm Div}_{d_1}(\mathbb D)+{\rm Div}_{d_2}(\partial\mathbb D)$ for some $d_1\geq 0, d_2\geq 1,  d_1+d_2=e$. 	 There are two equivalent ways to express $D$,  one is 
		$D=D_1+D_2$ where $D_1\in {\rm Div}_{d_1}(\mathbb D)$ and $D_2\in {\rm Div}_{d_2}(\partial\mathbb D)$, the other is $D=(B,S)$ where $B\in \mathcal B_{d_1, m}(={\rm Div}_{d_1}(\mathbb D))$ and $S\in {\rm Div}_{d_2}(\partial\mathbb D)$. The relation is $D_1=Z_B, D_2=S$. We use both ways in the paper without further explanation.
		
			 A sequence $(B_n)_n$ in ${\mathcal B_{e,m}}={\rm Div}_e({\mathbb D})$  {\it   converges to $D=(B,S)\in \partial {\mathcal B_{e,m}}=\partial {\rm Div}_e({\mathbb D})$ algebraically} (see \cite[\S 13]{Mc09b},  \cite[\S 1]{De05}),  denoted by $B_n\rightarrow D$, if the free zero divisors  $Z_{B_n}$ converge to 
		$Z_B+S$ in the   topology of ${\rm Div}_{e}(\overline{\mathbb D})$.  
	
	\begin{thm}  \label{bet}
		The map $\Psi_{e,m}$ extends to a homeomorphism 
		$$\Phi_{e,m}: {\rm Div}_e(\overline{\mathbb D})\rightarrow {\rm Div}_e(\overline{\mathbb D}).$$
		The extension is given as follows: write $D\in  \partial {\rm Div}_e({\mathbb D})$ as 
		$D=D_1+D_2$, where $D_1\in {\rm Div}_{d_1}(\mathbb D),  D_2\in {\rm Div}_{d_2}(\partial\mathbb D)$ such that 
		$d_1+d_2=e$, then 
		$$\Phi_{e,m}(D_{1}+D_2)=\Psi_{d_1,m}(D_1)+D_2,$$
		where $\Psi_{d_1, m}$ is the map given by Theorem \ref{z-p}.
	\end{thm}

	The proof is based on the following facts about the position of the critical points of a Blaschke product.
	
		\begin{thm} [Walsh \cite{W}] \label{walsh} The critical points of a finite Blaschke product are contained in the hyperbolic convex hull of the zeros.
			\end{thm}

	\begin{lem} \label{degenerate-0}   Let $B_n\in {\rm Div}_e({\mathbb D})=\mathcal B_{e,m}$ be a sequence of Blaschke products converging to $D=(B, S)\in \partial {\rm Div}_e({\mathbb D})$ algebraically.
		
		1.   If $1\notin {\rm  supp}(S)$, then the sequence  $(B_n)$ converges
		to $B$ in
		$\C\setminus {\rm  supp}(S)$.
		
		2.   If  $1\in {\rm  supp}(S)$, then  there exist  $\zeta\in \partial \mathbb D$ and a subsequence $(B_{n_k})_{k\geq 1}$, such that $B_{n_k}$ converges  to $\zeta B$ in 	$\C\setminus {\rm  supp}(S)$.
	\end{lem}
\begin{proof} 
The first statement follows from   \cite[Proposition 3.1]{Mc10}. 
The two statements  can be treated uniformly. 
 Write $S=\sum_{q\in {\rm supp}(S)} \nu(q)\cdot q$ and
 $$B_n(z)=z^m\Bigg(\prod_{k=1}^{e}\frac{z-a_{k,n}}{1-\overline{a_{k,n}}z}\Bigg) \Bigg(\prod_{k=1}^{e}\frac{1-\overline{a_{k,n}}}{1- {a_{k,n}}}\Bigg), \ a_{1,n}, \cdots, a_{e,n}\in \mathbb D  $$ 

	Note  that $\frac{z-a}{1-\bar{a}z}$ converges   to $-q$ in $\C\setminus\{q\}$  as $a\rightarrow q\in \partial \mathbb D$, and the inequality $|z_1\cdots z_m-w_1\cdots w_m|\leq \sum_{k=1}^m|z_k-w_k|$ for $z_k,w_k\in \overline{\mathbb D}$, we conclude that 
	\begin{itemize}
		\item If $1\notin {\rm  supp}(S)$, then  $B_n$ converges   in
	$\C\setminus {\rm  supp}(S)$ to 
	$$B \cdot \prod_{q\in  {\rm  supp}(S)}\bigg(\frac{1-\overline{q}}{1- q} (-q)\bigg)^{\nu(q)}=B.$$
	 	\item If $1\in {\rm  supp}(S)$, then there exist a subsequence $(n_j)_j$ and $\zeta\in \partial \mathbb D$ so that 
	 $$\lim_{j\rightarrow \infty}\prod_{k: \ a_{k,n_j}\rightarrow 1} (-1) \frac{1-\overline{a_{k,n_j}}}{1- {a_{k,n_j}}}=\zeta.$$
	  It follows that  $B_{n_j}$ converges to $\zeta B$ in 	$\C\setminus {\rm  supp}(S)$.
	 \end{itemize}
\end{proof}
	
		\begin{lem} \label{degenerate-1}   Let  $D=(B, S)\in \partial {\rm Div}_e({\mathbb D})$ with $1\in {\rm  supp}(S)$. For any  $\zeta\in \partial \mathbb D$, there is a   sequence  $B_n\in \mathcal B_{e,m}$ such that 
			$B_n\rightarrow D$ algebraically, and $B_{n}$ converges  to $\zeta B$ in 	$\C\setminus {\rm  supp}(S)$.
	\end{lem}
	\begin{proof} Write $S=\sum_{q\in {\rm supp}(S)} \nu(q)\cdot q$.  Since  $1\in {\rm  supp}(S)$, we may find a divisor sequence $X_n=\sum_{l=1}^{\nu(1)} 1\cdot a_{l,n}\in {\rm Div}_{\nu(1)}(\mathbb D)$ so that 
		$$X_n\rightarrow \nu(1)\cdot 1, \ \lim_{n\rightarrow \infty}(-1)^{\nu(1)}\prod_{l=1}^{\nu(1)} \frac{1-\overline{a_{l,n}}}{1- a_{l,n}}=\zeta.$$
		Let $(Y_n)_n$
 be a sequence of divisors in ${\rm Div}_{e-\nu(1)}(\mathbb D)$	so that 
$Y_n\rightarrow Z_B+S-\nu(1)\cdot 1$, where $Z_B$ is the free zero divisor of $B$.
Let  $B_n\in \mathcal B_{e,m}$ have free zero divisor    
$X_n+Y_n$.
By the same reasoning as that of Lemma \ref{degenerate-0}, we conclude that $B_{n}$ converges to $\zeta B$ in 	$\C\setminus {\rm  supp}(S)$.
	\end{proof}
	
Before the	proof of Theorem \ref{bet}, we introduce the following notations: For 
$D=\sum_{k=1}^{e} 1\cdot a_k\in  \partial {\rm Div}_e({\mathbb D})$ and $\epsilon>0$, define 
	$$N_\epsilon(D)=\Big\{\sum_{k=1}^{e} 1\cdot b_k; \
	b_k\in \mathbb D(a_k, \epsilon)\cap \mathbb D, \ 1\leq  k\leq e
	\Big\},$$
	$$U_\epsilon(D)=\Big\{\sum_{k=1}^{e} 1\cdot b_k; \
	b_k\in \mathbb D(a_k, \epsilon)\cap \overline{\mathbb D}, \ 1\leq  k\leq e
	\Big\}.$$

	\begin{proof}[Proof of Theorem \ref{bet}] Let $B_n\in \mathcal B_{e,m}$ be a sequence  converging to $D=(B, S)\in {\rm Div}_{d_1}(\mathbb D)\times  {\rm Div}_{d_2}(\partial\mathbb D)$ algebraically.
		By Lemma \ref{degenerate-0},  passing to choosing subsequence if necessary, there is  $\zeta\in \partial \mathbb D$, such that $B_{n}$ converges  to $\zeta B$ in 	$\C\setminus {\rm  supp}(S)$ (we set $\zeta=1$ if  $1\notin {\rm  supp}(S)$).

		By Weierstrass theorem, $B'_n$ converges to  $\zeta B'$ in 	$\C\setminus {\rm  supp}(S)$. 
		Note that $B$ has $d_1+m-1$ ciritical points in $\mathbb D$. Hence  $d_1+m-1$ critical points of 
		$B_n$ converges to that of $B$, and $d_2$ critical points of $B_n$ escape to the bounday $\partial \mathbb D$.

	In the following, we shall find out the    positions and multiplicity of the  degenerate critical points on $\partial \mathbb D$.
	Define the $\Psi_{e,m}$-impression $I(D)$ of   $D$:
		$$I(D)=\bigcap_{\epsilon>0}\overline{\Psi_{e,m}(N_\epsilon(D))}.$$
		Clearly, $I(D)$ is a connected and compact subset of $\partial{\rm Div}_e(\overline{\mathbb D})$. 
		
		By Theorem \ref{walsh}, the zero set 
		$(B'_n)^{-1}(0)$  is contained in the hyperbolic convex hull of $B_n^{-1}(0)$ for all $n$.
	Hence 
		 the  sequence of   free ramification divisors  $(R_{B_n})_n$ has only finitely many possible limits,  all  contained in the  finite set
		$$\Big\{\Psi_{d_1,m}(Z_B)+S';  \ S'\in {\rm Div}_{d_2}(\partial\mathbb D) \text{ and } {\rm supp}(S')\subset  {\rm supp}(S)\Big\}.$$
	  The connectivity of  $I(D)$ implies  that it is a singleton, say $\{R\}$.
			Write $Z_{B_n}=X_n+Y_n$ so that $X_n\in {\rm Div}_{d_1}(\mathbb D), Y_n\in {\rm Div}_{d_2}(\mathbb D)$ and $X_n\rightarrow Z_B$, $Y_n=\sum_{j=1}^{d_2} 1\cdot {b}_j(n)\rightarrow S=\sum_{j=1}^{d_2} 1\cdot {b}_j $.  We may assume $ {b}_j(n)\rightarrow b_j$ for each $j$. 

		To get $R$, we  evaluate the limit $R=\lim_{n\rightarrow \infty}\Psi_{d, m}(Z_{B_n})$ by repeated limit: 
		\bess R&=&\lim_{X_n\rightarrow Z_B}\lim_{b_1(n)\rightarrow b_1} \cdots \lim_{b_{d_2}(n)\rightarrow b_{d_2}}\Psi_{d, m}\Big(X_n+\sum_{j=1}^{d_2} 1\cdot b_j(n)\Big)\\
		&=&\lim_{X_n\rightarrow  Z_B}\lim_{b_1(n)\rightarrow b_1} \cdots \lim_{b_{d_2-1}(n)\rightarrow b_{d_2-1}}\Psi_{d-1,m}\Big(X_n+\sum_{j=1}^{d_2-1} 1\cdot b_j(n)\Big)+1\cdot b_{d_2}\\
		&=&\cdots=\lim_{X_n\rightarrow  Z_B}\Psi_{d_1,m}\big(X_n\big)+\sum_{j=1}^{d_2} 1\cdot b_j
	=\Psi_{d_1,m}\big(Z_B\big)+S.
		\eess
		This gives the extension $\Phi_{e,m}(D)=\Psi_{d_1,m}\big(Z_B\big)+S$. 
		One may verify that   $\Phi_{e,m}$ is continuous, bijective, and the inverse $\Phi^{-1}_{e,m}$ is also continuous. Hence  $\Phi_{e,m}$  is a homeomorphism. 
	\end{proof}
	
	\begin{ex} When $e=1$, ${\rm Div}_1(\overline{\mathbb D})=\overline{\mathbb D}$, the map $\Phi_{1,m}: \overline{\mathbb D}\rightarrow \overline{\mathbb D}$  has 
		formula:
		$$\Phi_{1,m}(a)=\frac{2am}{(m-1)|a|^2+(m+1)+\sqrt{[(m-1)|a|^2+(m+1)]^2-4m^2|a|^2}}.$$
		In particular, when $m=1$,
		$$\Phi_{1,1}(a)=\frac{a}{1+\sqrt{1-|a|^2}}, \ a\in \overline{\mathbb D}.$$
	Clearly
		$\Phi_{1,m}|_{\partial{\mathbb D}}$ is the identity map.
	\end{ex}

	\section{Continuity properties} \label{continuity-property}

	
	For a rational map $f$,  let $J(f)$ and $F(f)$ denote the Julia set and the Fatou set.
	Each component of $F(f)$ is called a {\it Fatou component}.  The Fatou  component   containing   $a\in F(f)$ is denoted by $U_f(a)$.  When $f$ is a polynomial, we use 
	$K(f)$ to denote the filled Julia set. 
	
	\vspace{5pt}
\noindent \textbf{Polynomial dynamics.}	Let $\mathcal C_d=\{f\in \mathcal P_d; J(f) \text{ is  connected}\}$ be the connectedness  locus.  It's known that $\mathcal C_d$ is compact and connected (see  \cite{DH,DeMP}).
For any $f\in\mathcal{C}_d$, 
 there is a unique conformal map   $\psi_{f,\infty}:\mathbb{C}\setminus K(f)\rightarrow\mathbb{C}\setminus \overline{\mathbb{D}}$   tangent to the identity at $\infty$ and satisfying that $\psi_{f,\infty}(f(z))=\psi_{f,\infty}(z)^d$  \cite[\S 9]{Mil06}. This $\psi_{f,\infty}$ is called the {\it B\"ottcher map} of $f$ at $\infty$.  For each  $\theta\in\mathbb{R/Z}$, the \emph{external ray} $R_f(\theta)$ is defined by $R_f(\theta) = \psi_{f,\infty}^{-1}((1,\infty)e^{2\pi i\theta})$.  It satisfies $f(R_f(\theta))=R_f(d\theta)$.
 We say $R_f(\theta)$ {\it lands } at $a\in J(f)$ if $\lim_{r\rightarrow 1^+} \psi_{f,\infty}^{-1}(re^{2\pi i\theta})=a$. 
	

	Let $K\subset\mathbb{C}$ be a full connected compact set with a Jordan domain $U\subset K$.  We say $K$ admits a {\it limb 
		decomposition} with respect to $U$ if 
	$$K=U\bigsqcup \bigsqcup_{x\in \partial U} L_{U, x},$$ 
	where $L_{U, x}$ is a connected compact set and 
	$L_{U, x}\cap\overline{U}=\{x\}$   for each  $x\in \partial U$. 

	\begin{thm}  [\cite{RY22}]\label{RY} 
		Let $f\in \mathcal C_d$  and let $U$ be a pre-attracting or pre-parabolic bounded Fatou component of $f$. 
		Then the following properties hold. 
		\begin{enumerate}
			\item
			\label{U-Jordan}
			$U$ is a Jordan domain. 
			
			\item $K(f)$ admits  a   limb decomposition $K(f)=U\bigsqcup \bigsqcup_{x\in \partial U} L_{U, x}$ with respect to $U$. 
			
			\item 
			\label{RY-separate}
			If $L_{U, x}=\{x\}$, there is only one external ray landing at $x$; if $L_{U, x}\neq\{x\}$, there are two external rays landing at $x$ and separating $L_{U, x}$ from $U$. 
			
			\item 
			\label{dynam-prop-crit}
			For any $x\in \partial U$, the limb $L_{U, x}$ is not reduced to $\{x\}$ if and only if there is an integer $n\geq 0$ such that $L_{f^n(U), f^n(x)}$ contains a critical point.
		\end{enumerate}
	\end{thm}
	
	Let $f$ and $U$ be as in Theorem \ref{RY}.
For each $y\in K(f)\setminus U$, there is a unique point $x\in \partial U$ so that $y\in L_{U, x}$.  This induces  a natural projection
	\begin{equation} \label{proj-u}
		\sigma_U:
		\begin{cases} K(f)\setminus U\rightarrow\partial U\\
			y\mapsto x
		\end{cases}.
	\end{equation}

	For each $x\in  \partial U$, if $L_{U, x}=\{x\}$, we denote the unique external ray landing at $x$ by $R_f(\theta)$, and set $\theta_U^+(x)=\theta_U^-(x)=\theta$; if   $L_{U, x}\supsetneq \{x\}$, there are two different external rays, say $R_f(\alpha), R_f(\beta)$ landing at $x$ so that $R_f(\alpha), L_{U, x}, R_f(\beta)$ attach at $x$ in counterclockwise order. We  set $\theta_U^+(x)=\beta, \theta_U^-(x)=\alpha$.  
	
	In this way, we get two maps $\theta_U^{\pm}: \partial U\rightarrow \mathbb R/\mathbb Z$.

	\begin{lem} \label{cont-limb}  We have the following assertions.
		
		(1).  The map $\sigma_U:
		K(f)\setminus U\rightarrow\partial U$  is continuous.
		
		(2).  Let $(x_n)_n$ be a sequence  in $\partial U$,   $x\in \partial U$. 
		If  $(x_n)_n$ converges  to $x$ in clockwise order, then 
		$$\lim_n\theta_U^+(x_n)=\theta_U^+(x), \ \lim_n\theta_U^-(x_n)=\theta_U^+(x).$$  
		If  $(x_n)_n$ converges  to $x$  in counterclockwise order, then 
		$$\lim_n\theta_U^+(x_n)=\theta_U^-(x), \ \lim_n\theta_U^-(x_n)=\theta_U^-(x).$$  
		
		In particular, $\theta_U^{\pm}$  is  continuous at $x\in \partial U$ if and only if $L_{U,x}=\{x\}$.
		
		(3).  $a\in  \partial  U$ is a cut point of $J(f)$ if and only if $L_{U, a}\neq \{a\}$.
	\end{lem}
  
  Here $a\in J(f)$ is called  a {\it cut point} of $J(f)$, if $J(f)\setminus \{a\}$ is disconnected.
  Lemma \ref{cont-limb}   is an immediate consequence  of Theorem \ref{RY},  so we omit its proof.
	
	
	\begin{cor}\label{limb-arc}  Suppose  $L_{U, x}=\{x\}$ for some $x \in \partial U$.
		
		(1).  For any shrinking sequence $(C_n)_n$ of arcs in $\partial U$ with $\bigcap_n C_n =\{x\}$, we have ${\rm diam}(\sigma_U^{-1}(C_n))\rightarrow 0$.
		
		(2).  $J(f)$ is locally connected at $x$. 
	\end{cor}
	\begin{proof} (1). Replacing $C_n$ with $\overline{C_n}$, we  assume $C_n$ is a closed set.
		By  Lemma \ref{cont-limb} (1),  $( \sigma_U^{-1}(C_n))_n$ is a sequence of shrinking compact sets.  By the equality
		\begin{equation}\label{intersection-sigma} \bigcap_n \sigma_U^{-1}(C_n)=\sigma_U^{-1}\Big(\bigcap_n C_n \Big)=\sigma_U^{-1}(x)=L_{U,x}
			\end{equation}
		and the assumption  $L_{U, x}=\{x\}$, 
		we get   ${\rm diam}(\sigma_U^{-1}(C_n))\rightarrow 0$.

		(2).  Let $C_n$ be the component of $\mathbb D(x, 1/n)\cap \partial U$ containing $x$. 
		Then $(C_n)_n$ is a sequence of open arcs with ${\rm diam}(C_n)\rightarrow 0$. 
		By Lemma \ref{cont-limb} (1),  the restriction $\sigma_U|_{J(f)}$ is continuous, hence  $\sigma_U|_{J(f)}^{-1}(C_n)=\bigcup_{x\in C_n}(L_{U, x}\cap J(f))$  is an open subset of $J(f)$.  Clearly $\sigma_U|_{J(f)}^{-1}(C_n)$ is connected. 
		By (1), we have  ${\rm diam}(\sigma_U^{-1}(C_n))\rightarrow 0$.
		Therefore  $\{\sigma_U|_{J(f)}^{-1}(C_n)\}_n$ gives a basis of open and connected neighborhoods of $x$, implying the local connectivity of $J(f)$ at $x$.
	\end{proof}


	\noindent \textbf{Kernel convergence.}	
		A disk is a simply connected domain in $\mathbb C$.  Let $\mathcal D$ be the set of pointed disks $(U,u)$. The {\it Carath\'eodory topology} or {\it kernel convergence} on $\mathcal D$ is defined as follows: $(U_n, u_n)\rightarrow (U, u)$ if and only if 
	
	(i).  $u_n\rightarrow u$;
	
	(ii).  for any compact $K\subset U$, $K\subset U_n$ for all $n$ sufficiently large; and 
	
	
	(iii). for $w\in \partial U$, there exist $w_n\in \partial U_n$ such that $w_n\rightarrow w$ as $n\rightarrow +\infty$.

	Let $\mathcal E\subset \mathcal D$ denote  the subspace of disks not equal to $\mathbb C$.
	
	Let $f_n: (U_n, u_n) \rightarrow \mathbb C$ be a sequence of holomorphic maps. Following McMullen \cite[\S 5.1]{Mc94}, we say that  $f_n$ converges to $f:(U, u)\rightarrow \mathbb C$ in {\it Carath\' eodory topology on functions}  if
	
	(i). $(U_n, u_n)\rightarrow (U, u)$ in $\mathcal D$, and
	
	(ii). for any compact $K\subset U$ and  large  $n$,  $f_n|_K$ converges  uniformly to $f|_K$.
	
	\vspace{5pt}
	
	In our discussion, a Riemann (or conformal)  mapping $f:(\mathbb D,0)\rightarrow (U, u)$ is a biholomorphic map $f:\mathbb D\rightarrow U$ with $f(0)=u$. 
	
	The following is well-known, see \cite{Car}, \cite[\S 5.1]{Mc94}.
	
	\begin{thm}\label{cara-riemann}  Let $(U_n ,u_n), (U, u)$ be in $\mathcal E$. Let  $f_n:(\mathbb D,0)\rightarrow (U_n, u_n)$ and 	$f:(\mathbb D,0)\rightarrow (U, u)$ be Riemann mappings with $f_n'(0)>0$ and $f'(0)>0$. Then 
		
		(1).  $(U_n ,u_n)\rightarrow (U, u)$ if and only if $f_n$ converges  to $f$ in $\mathbb D$; 
		
		(2). If $(U_n ,u_n)\rightarrow (U, u)$, then $f_n^{-1}\rightarrow f^{-1}$ in Carath\' eodory topology on functions. 		
	\end{thm}
	
	\begin{rmk}\label{cara-function} In Theorem \ref{cara-riemann}, assume $(U_n ,u_n)\rightarrow (U, u)$,  if  $f_n$  is not normalized so that $f_n'(0)>0$, then  the statement reads as: there exist a  Riemann   mapping $g:(\mathbb D,0)\rightarrow (U, u)$  and a subsequence $(f_{n_k})_k$
		so that 	
		
		(1). $f_{n_k}$ converges to $g$ in $\mathbb D$; 
		
		(2). $f_{n_k}^{-1}\rightarrow g^{-1}$ in Carath\' eodory topology on  functions. 		
	\end{rmk}
	
	 
	The technique of  utilizing hyperbolic metrics  in  the  kernel convergence of pointed disks appears in  Luo’s work  \cite[\S 6]{Luo24}  to study the limits of quasi-invariant trees, Petersen-Zakeri's  work \cite[\S 2.4]{PZ24} on  Hausdoff limits of external rays.
	A notable property is that in  the kernel convergence,  the hyperbolic distance  descends to the limit:
	\begin{lem} \label{hyp-d-limit} Assume $(U_n ,u_n)\rightarrow (U, u)$ in $\mathcal E$.  
		
		(1).		Suppose $a_n, b_n\in U_n$, $a,b\in U$ satisfy that   $a_n\rightarrow a$, $b_n\rightarrow b$. Then we have the convergence of the hyperbolic distances
		$$d_{U_n}(a_n, b_n)\rightarrow d_{U}(a, b).$$
		
		(2).		Suppose $a_n, b_n\in U_n$   satisfy that   $a_n\rightarrow a\in U$, $b_n\rightarrow b\in \mathbb C$,  then 
		$$b\in U   \ \text{ if and only if }  \  \sup_{n} d_{U_n}(a_n, b_n)<+\infty.$$				
	\end{lem}
	\begin{proof}  Let $f_n: \mathbb D\rightarrow U_n$ be the Riemann mapping so that $f_n(0)=a_n, f_n(r_n)=b_n$, where $r_n>0$ is chosen so that $d_{\mathbb D}(0, r_n)=d_{U_n}(a_n, b_n)$.
		By Theorem \ref{cara-riemann} and Remark \ref{cara-function},  also by passing to a subsequence, $f_n$  converges  to a conformal map  $g:  \mathbb D\rightarrow U$ with $g(0)=a$, and 
		$f_n^{-1}\rightarrow g^{-1}$ in Carath\' eodory topology on functions.

		(1).			  Since $b_n\rightarrow b\in U$, we get $r_n=f_n^{-1}(b_n)\rightarrow r_g:= g^{-1}(b)$. Hence
		$$d_{U_n}(a_n, b_n)=\log\frac{1+r_n}{1-r_n}\rightarrow \log\frac{1+r_g}{1-r_g}=d_U(a,b).$$
		
		(2).	  If $b\in U$, by (1), $d_{U_n}(a_n, b_n)\rightarrow d_{U}(a, b)$ and $\sup_{n} d_{U_n}(a_n, b_n)<+\infty$.
		Conversely, assume $\sup_{n} d_{U_n}(a_n, b_n)\leq L$ for some $L\geq 0$,
		then $r_n\leq r:=(e^L-1)/(e^L+1)<1$ for all $n$.  Assume $r_n\rightarrow r_\infty\leq r$,  by the uniform convergence of $f_n$ to $g$ in the closed disk $\overline{\mathbb D(0,r)}$, we have 
		$b_n=f_n(r_n)\rightarrow g(r_\infty)$. It follows that $b=g(r_\infty)\in g(\mathbb D)=U$.
	\end{proof}
	
		\vspace{3pt}
	\noindent \textbf{Kernel convergence arising from dynamics.}	
	We say  a sequence of rational maps  $(f_n)_{n}$ converges to $f$ algebraically if ${\rm deg}(f_n)={\rm deg}(f)$ and the coefficients of $f_n$   can be chosen to converge to those of $f$.
	
	
	\begin{lem} \label{kernel-att}
	Let $(f_n)_n  $ be a sequence of rational maps converging to $f$ algebraically. Assume that  each $f_n$ has an attracting  fixed point $a_n$, and $a_n\rightarrow a$, which is    an $f$-attracting fixed point. Assume the  Fatou components $U_{f_n}(a_n), U_{f}(a)$ are simply connected. 
	Then we have the kernel convergence
	$$(U_{f_n}(a_n), a_n)\rightarrow (U_f(a), a).$$
	\end{lem}
	\begin{proof} We check the definition of kernel convergence. (ii) is immediate. 
		(iii) is due to the density of   repelling periodic points on Julia set and their stability. 
		\end{proof}
	\begin{rmk}\label{kernel-att-pre} Under the condition of Lemma \ref{kernel-att}, if $f_n^{l}(b_n)=a_n$ for some integer $l\geq 1$ and for all $n$, and if $b_n\rightarrow b$, $U_{f_n}(b_n)$ and $U_f(b)$ are simply connected,  we also have  the kernel convergence:
		$$(U_{f_n}(b_n), b_n)\rightarrow (U_f(b), b).$$
		\end{rmk}







A sequence of  compacta $(E_n)_n$ converges to a compactum $E$ in {\it Hausdorff topology}  if 
$d_{H}(E_n, E)\rightarrow 0$, where $d_H$ is the Hausdorff distance  defined by
$$d_H(A, B)=\max\Big\{ \max_{a\in A}\min_{b\in B} d(a,b),  \  \max_{b\in B}\min_{a\in A} d(a,b)\Big\},$$
and $d(a,b)$ is the Euclidean or spherical distance  depending on the situation.


\begin{lem} \label{c-internal-ray} Let $(f_n)_n$ be a  sequence in $\mathcal C_d$ converging to $f$.   Let $(U_n, a_n), (U, a)$ be pointed bounded  attracting or parabolic Fatou components of $f_n, f$ respectively.
	Let $p\in \partial U$ be a repelling periodic point of $f$.
Assume the kernel convergence 
	$$(U_n, a_n)\rightarrow (U, a).$$
	Then there exist arcs $\gamma_n: [0,1]\rightarrow \overline{U_n}$, $\gamma: [0,1]\rightarrow  \overline{U}$ with the properties  
	
	\begin{itemize}
		 \item $\gamma_n([0,1))\subset U_n$,  $\gamma_n(1)\in \partial  U_n$ is $f_n$-repelling;
		  $\gamma([0,1))\subset U, \gamma(1)=p$;
		  
		  \item $\gamma_n(0)=\gamma(0)$ for $n$ large enough;
		 
		 \item $\gamma_n\rightarrow\gamma$ in  Hausdorff topology. 
		\end{itemize}
\end{lem}

\begin{proof} 
Suppose the $f$-period of $p$ is $l\geq 1$.	By the implicit function theorem,   there exist a neighborhood $\mathcal N$ of $f$ and a   continuous map $r: \mathcal N\rightarrow\mathbb C$ with $r(f)=p$,  so that $r(g)$ is $g$-repelling for all $g\in \mathcal N$.  By shrinking $\mathcal N$ if necessary, we can find a common linearization neighborhood  $V$  of $g^l$ near $r(g)$ for all $g\in \mathcal N$.
There is  a fundamental arc $\alpha_g\subset V$ which generates a $g^l|_V^{-1}$-invariant curve $\gamma_g$  converging to $r(g)$.  By shrinking $\mathcal N$,  
we may further require that
\begin{itemize}
	\item  the family of arcs $\{\alpha_g\}_{g\in \mathcal N}$ have a common starting point;
	
	\item $\alpha_g$ is continuous with respect to $g\in \mathcal N$ in Hausdorff topology; 
	
	\item $\alpha_f\subset U$.
		\end{itemize}
	
	It follows that $\mathcal N\ni g\mapsto \gamma_g$ is Hausdorff continuous. By the kernel convergence $(U_n, a_n)\rightarrow (U, a)$,  we have that $f_n\in \mathcal N$ and  $\alpha_{f_n}\subset U_n$ for all large $n$.  Therefore $\gamma_n:=\gamma_{f_n}\subset U_{n}\cup\{r(f_n)\}$ and the conclusion follows.
	\end{proof}

	\begin{lem} \label{continuity-rays} Let $f\in \mathcal C_d$ and let $U$ be a bounded attracting  Fatou component of $f$.  Suppose that $R_f(\theta)$ lands at  $\xi\in \partial U$ and  $\xi$ is not a cut point of $J(f)$.  
		Then for any sequence of maps $(f_n)_n\subset\mathcal C_d$ and any sequence of angles $(\theta_n)_n$ with  $f_n\rightarrow f$ and $\theta_n\rightarrow \theta$, we have the Hausdorff convergence (in spherical metric)		$$ \overline{R_{f_n}(\theta_n)}\rightarrow  \overline{R_{f}(\theta)}.$$
 	\end{lem}
 Note that we don't assume   the external ray $R_{f_n}(\theta_n)$ lands for each $n$.
 \begin{proof}   Since   $\xi\in \partial U$ is not a cut point of $J(f)$, we have $L_{U, \xi}=\{\xi\}$ (see Lemma \ref{cont-limb}(3)).  By Corollary \ref{limb-arc}, for
 	 any $\epsilon>0$,   there is  an open arc $C\subset \partial U$  containing $\xi$ and satisfying that 
\begin{itemize}     
	\item the two endpoints $a, b$ of $C$ are repelling periodic points of $f$.
	\item ${\rm diam}(\sigma_U^{-1}(C))\leq \epsilon$, where $\sigma_U$ is defined by \eqref{proj-u}.
	\end{itemize}

Let $\alpha$ be the attracting  periodic point in $U$.
By the stability of  attracting point,  there is an attracting point $\alpha_n$ of $f_n$ with $\alpha_n\rightarrow \alpha$. By Lemma \ref{kernel-att},  we have the kernel convergence   $(U_{f_n}(\alpha_n), \alpha_n)\rightarrow (U, \alpha)$. By Lemma \ref{c-internal-ray},  for $\omega\in\{a,b\}$ and for each $n$,  there exist an arc $\gamma_{\omega,n}: [0,1]\rightarrow \overline{U_{f_n}(\alpha_n)}$  so that

\begin{itemize}
	\item $\gamma_{\omega,n}([0,1))\subset U_{f_n}(\alpha_n)$,  $\gamma_{\omega,n}(1)\in \partial   U_{f_n}(\alpha_n)$ is $f_n$-repelling;
	$\gamma_{\omega}([0,1))\subset U, \gamma_{\omega}(1)=\omega$;
	
	\item $\gamma_{\omega,n}(0)=\gamma_\omega(0)$ for all $n$;
	
	\item $\gamma_{\omega,n}\rightarrow\gamma_\omega$ in  Hausdorff topology. 
\end{itemize}

For $\omega\in\{a,b\}$,  there is an external ray, say $R_f(\theta_\omega)$, landing at $\omega$ (see \cite[Theorem 18.11]{Mil06}). Set $\zeta_\omega=\gamma_\omega(0)$. 
 Let $\beta\subset U$ be an arc connecting $\zeta_a$ and $\zeta_b$. 
By suitable choices  of $\gamma_a, \gamma_b$ and  $\beta$, we may assume ${\rm diam}(\gamma_a\cup\gamma_b\cup \beta)\leq 2\epsilon$.
 Let 
 $$X_n=\overline{R_{f_n}(\theta_a)}\cup \gamma_{a,n}\cup \overline{R_{f_n}(\theta_b)}\cup  \gamma_{b,n}, \ X=\overline{R_{f}(\theta_a)}\cup \gamma_{a}\cup \overline{R_{f}(\theta_b)}\cup  \gamma_{b}.$$

The assumption  $f_n\rightarrow f$ and $\theta_n\rightarrow \theta$ implies that for large $n$, the set $\overline{R_{f_n}(\theta_n)}$ is  in the  component of $\mathbb C\setminus X_n\cup \beta$ containing $R_{f}(\theta)$.  By the Hausdorff convergence $X_n\rightarrow X$, we conclude that 
  $\overline{R_{f}(\theta)}$  and the accumulation set of  $(\overline{R_{f_n}(\theta)})_n$  differ by a set with diameter no larger than 
  $$ {\rm diam}(\sigma_U^{-1}(C))+ {\rm diam}(\gamma_a\cup\gamma_b\cup \beta)\leq 3\epsilon.$$
Since $\epsilon$ is arbitrary, we get the Hausdorff convergence.
 	\end{proof}
 
 Lemma \ref{continuity-rays} can be generalized to the following situation, which is applicable to the parabolic case.
 
 	\begin{lem} \label{continuity-rays-general}  Let $(f_n)_n\subset \mathcal C_d$ converge to $f\in \mathcal C_d$.   Let $(U_n, a_n), (U, a)$ be given in Lemma \ref{c-internal-ray}. Suppose that $R_f(\theta)$ lands at  $\xi\in \partial U$ and $\xi$ is not  a cut point of $J(f)$.  For  any sequence of angles $(\theta_n)_n$  with   $\theta_n\rightarrow \theta$, we have the Hausdorff convergence (in spherical metric)		$$ \overline{R_{f_n}(\theta_n)}\rightarrow  \overline{R_{f}(\theta)}.$$
 \end{lem}
 The proof of Lemma \ref{continuity-rays-general} is same as that of Lemma \ref{continuity-rays}. We omit the details.
 
 	
 	\vspace{5pt}
 \noindent \textbf{Continuity of radial rays.}	
  The following  Proposition \ref{convergent-rays} proves the continuity of most radial rays for a sequence of convergent holomorphic maps
  with uniformly bounded $L^2$-derivatives. 
  Proposition \ref{kernel-conformal} is one of its applications.
 
 	\begin{pro} \label{convergent-rays} Let $A=\{r<|z|<R\}$ be an annulus. 
 	Let $f_n: A\rightarrow \mathbb C$ be a  sequence of holomorphic maps converging to $f: A\rightarrow \mathbb C$.  Assume that 
 	$$\sup_n \int_{A}|f_n'(z)|^2dxdy<+\infty.$$
 	For $\theta\in [0,2\pi]$ and $g\in \{f_n, f\}$, define the length function
 	$$L_g:   
 	\begin{cases}  [0,2\pi]\rightarrow (0, +\infty], \\
 		\theta \mapsto \displaystyle\int_r^R |g'(\rho e^{i\theta})|d\rho.
 	\end{cases}$$
 	
 	
 	(1).  $L_{f_n}, L_f $ are in  $L^1[0, 2\pi]$, and  we have the $L^1$-convergence:
 	$$\lim_{n\rightarrow \infty}\int_{0}^{2\pi }|L_{f_n}(\theta)-L_f(\theta)| d\theta=0.$$

 	(2). There exist a full measure set $E$ of $[0,2\pi]$, and a subsequence $(f_{n_k})_k$  of $(f_n)_n$ satisfying that
 	
 	(a). For any  $\theta\in E$ and any $g\in \{f, f_{n_k}; k\geq 1\}$, the following limits exist: 
 	$$\lim_{\rho\rightarrow R^-}g(\rho e^{i\theta}), \ \lim_{\rho\rightarrow r^+}g(\rho e^{i\theta}).$$ 
 	
 	(b). For any  $\theta\in E$, the sequence $(f_{n_k})_k$ converges uniformly to $f$ on $[r,R]e^{i\theta}$.
 \end{pro}
 \begin{proof}  Write $\|g\|= (\int_{A}|g(z)|^2dxdy )^{1/2}$ for a holomorphic map $g:A\rightarrow \mathbb C$.
 	Let $M=\sup_{n} \|f'_n\|$. 
 	Since $f_n$ converges  to $f$ in $A$,  we  get $\|f'\|\leq M$.
 		 
		By Cauchy-Schwarz, for $g=f_n$ or $f$, 
		$$\bigg(\int_{0}^{2\pi }L_{g}(\theta)d\theta\bigg)^2\leq  2\pi\log(R/r)  \|g'\|^2\leq 2\pi M^2\log(R/r). $$
		Hence $L_g\in L^1[0, 2\pi]$ and  $E_g:=\{\theta\in[0,2\pi]; L_g(\theta)<+\infty\}$ has full measure.
		
		Choose $r<r'<R'<R$, then  
		\bess&&\int_{0}^{2\pi }|L_{f_n}(\theta)-L_f(\theta)| d\theta\leq \int_{0}^{2\pi }\int_{r}^R |f'_n-f'| d\rho d\theta\\
		&=&  \underbrace{\int_{0}^{2\pi }  \int_{r}^{r'} |f'_n-f'| d\rho d\theta}_{I_1}+
		 \underbrace{ \int_{0}^{2\pi }  \int_{r'}^{R'} |f'_n-f'| d\rho d\theta}_{I_2}+
		 \underbrace{   \int_{0}^{2\pi }  \int_{R'}^{R} |f'_n-f'| d\rho d\theta}_{I_3}.
			\eess
			
		By Cauchy-Schwarz again,
		\bess
		I_1^2&\leq& 2\pi \log(r'/r) \|f'_n-f'\|^2\leq 8\pi M^2  \log(r'/r),\\
			I_3^2&\leq& 2\pi \log(R/R') \|f'_n-f'\|^2\leq 8\pi M^2 \log(R/R').
		\eess
		For any $\epsilon>0$, choose $R'$ sufficiently close to $R$, and $r'$ sufficiently close to $r$, so that $I_1 \leq \epsilon, I_3\leq \epsilon$. For the chosen $r'$ and $R'$, since $f_n$ converges uniformly in $\{r'\leq |z|\leq R'\}$ to $f$, by Weierstrass's Theorem, 
		 there is an integer $N>0$ so that $I_2\leq \epsilon$ for $n\geq N$. If follows that $\int_{0}^{2\pi }|L_{f_n}-L_f| d\theta\leq 3\epsilon$ for $n\geq N$, establishing the $L^1$-convergence.
		
		(2). Let $E_0=E_f\bigcap \bigcap_n E_{f_n}$. Then $E_0$ is a full measure subset of $[0,2\pi]$. Moreover, for any $g\in \{f_n, f; n\geq 1\}$ and  any $\theta \in E_0$, we have $L_g(\theta)<\infty$, this implies that the limits $\lim_{\rho\rightarrow R^-}g(\rho e^{i\theta})$,  $\lim_{\rho\rightarrow r^+}g(\rho e^{i\theta})$ exist.
		
		Define $L_g^{s}(\theta)=\displaystyle\int_{rs}^{R/s} |g'(re^{i\theta})|dr$
	for  $s\in (1, \sqrt{R/r})$.   
	By the $L^1$-convergence, there is a subsequence $(f_{n_k})_k$ of $(f_n)_n$ and a full measure subset $E$ of $E_0$ so that $L_{f_{n_k}}(\theta)\rightarrow L_f(\theta)$  for any $\theta\in E $.  Hence for the given $\theta\in E $ and for any $\epsilon>0$, there is a number $s\in (1, \sqrt{R/r})$ and independently a positive integer $k_1$ so that
		$$ \ L_f(\theta)-L_f^{s}(\theta)\leq \epsilon; \ |L_{f_{n_k}}(\theta)-L_{f}(\theta)|\leq \epsilon, \ \forall k\geq k_1.$$    By the uniform convergence $f_{n_k}\rightarrow f$ in $A_s:=\{rs\leq |z|\leq R/s\}$,   there is  $k_2\geq k_1$ so that $|L_{f_{n_k}}^{s}(\theta)-L_{f}^{s}(\theta)|\leq \epsilon$ for $k\geq k_2$. It follows that 
		$$L_{f_{n_k}}(\theta)-L_{f_{n_k}}^{s}(\theta)\leq |L_{f_{n_k}}(\theta)-L_{f}(\theta)|+|L_{f}(\theta)-L_{f}^{s}(\theta)|+|L_{f}^{s}(\theta)-L_{f_{n_k}}^{s}(\theta)|\leq 3\epsilon.$$
		
		Choose $k_3\geq k_2$ so that $\max_{z\in A_s}|f_{n_k}(z)-f(z)|\leq \epsilon$ for $k\geq k_3$.
		For any $\rho\in [r,rs]\cup[R/s, R]$, 
		$$|f_{n_k}(\rho e^{i\theta})-f(\rho e^{i\theta})|\leq
		L_{f_{n_k}}(\theta)-L_{f_{n_k}}^{s}(\theta)+ L_{f}(\theta)-L_{f}^{s}(\theta) +\epsilon \leq 5\epsilon.$$
		The uniform convergence follows.
		\end{proof}
	
		\begin{rmk} (1). In Proposition \ref{convergent-rays}, the annulus $A$ can be replaced by the disk $\mathbb D$ without changing the idea of the proof. 
			
		(2).	If all $f_n$ are univalent, then $\|f'_n\|^2={\rm area}(f_n(A))$. In this case,  the uniform boundness of $L^2$-derivatives has the geometric meaning 
			$$\sup_{n} {\rm area}(f_n(A))<+\infty.$$	
 	\end{rmk}

		\begin{pro} \label{kernel-conformal}
		Let $(f_n)_n  $ be a sequence of polynomials in $\mathcal C_d$  converging to $f$. 
		Let $(U_n, a_n), (U, a)$ be pointed bounded  attracting or parabolic Fatou components of $f_n, f$ respectively. Let $\phi_{n}: (\mathbb D, 0)\rightarrow (U_n, a_n)$  and  $\phi: (\mathbb D, 0)\rightarrow (U, a)$ be conformal maps\footnote{By Theorem \ref{RY} and Carath\'eodory's boundary extension theorem, $\phi_{n}$ and $\phi$ can extend to   homeomorphisms between the closures of their domains and ranges.   So it is meaningful to write $\phi_n(\zeta), \phi(\zeta)$ when $\zeta\in \partial \mathbb D$.}. Assume that $\phi_n$  converges to $\phi$ in $\mathbb D$.

			(1).  Let
		$(q_n)_n$  be a sequence  in $\overline{\mathbb D}$ converging to $q\in \partial\mathbb D$.  
		
		\begin{itemize}
			\item  If  $\phi(q)\in \partial U$ is not a cut point of $J(f)$, then
			$$\lim_{n\rightarrow \infty}\phi_n(q_n)=\phi(q).$$
			
			\item   If  $\phi(q)\in \partial U$ is  a cut point of $J(f)$, then any accumulation point of 
			the sequence $(\phi_n(q_n))_n$ is contained in $L_{U, \phi(q)}$.
			\end{itemize}
			In particular, $\phi_n$ converges pointwisely to $\phi$ in the following subset of $\partial \mathbb D$: 
		$$\{q\in \partial \mathbb D; \phi(q) \text{ is not a cut point of  }J(f)\}.$$

	(2).  Let
$(q_n)_n$  be a sequence in  $\partial \mathbb D$  converging to $q\in \partial\mathbb D$. For each $n$, let $R_{f_n}(\theta_n)$ be an external ray landing at $\phi_n(q_n)$ \footnote{The existence of such  $R_{f_n}(\theta_n)$ is guaranteed by Theorem \ref{RY}.}.   Then
$$\Big( \bigcap_{k\geq 1}\overline{\bigcup_{n\geq k}  {R_{f_n}(\theta_n)}}\Big) \bigcap K(f)\subset L_{U, \phi(q)}. $$

			\end{pro}
		
		\begin{proof}  
			Note that $\|\phi'_n\|^2={\rm area}(U_n)\leq \pi$ for all $n$.			
				By Proposition \ref{convergent-rays} and also by choosing a subsequence, there is a full measure set $E$ of $[0,2\pi]$, such that for any $\theta\in E$, the sequence $(\phi_{n})_n$ converges uniformly to $\phi$ on $[0,1]e^{i\theta}$.

			For any $\epsilon>0$, there is an   arc $\Gamma_\epsilon\subset \partial \mathbb D$ whose interior contains $q$ so that
			\begin{itemize}
				\item the two endpoints $\xi, \zeta$ of $\Gamma_\epsilon$ are contained in $\{e^{i\theta}; \theta\in E\}$, and  $\phi(\xi), \phi(\zeta)$  are not cut points of $J(f)$;
				
				\item the $\phi$-image $C_\epsilon=\phi(\Gamma_\epsilon)$ is contained in  $\mathbb D(\phi(q), \epsilon)$.
				\end{itemize}

By Theorem \ref{RY},  there are  unique external rays $R_f(\alpha)$ and $R_f(\beta)$   landing at $\phi(\xi)$ and $\phi(\zeta)$ respectively; moreover,  there are external rays    $R_{f_n}(\alpha_n)$ and $R_{f_n}(\beta_n)$  landing at  $\phi_{n}(\xi)$ and $\phi_{n}(\zeta)$ respectively, for each $n$. 
			Note that $\alpha_n$ (or $\beta_n$) might be not unique, and  we choose one of them.
			
			\vspace{5pt}
	 \textbf{Claim: } 	{\it  $\lim_n \alpha_n=\alpha$ and  $\lim_n \beta_n=\beta$.}
			\vspace{5pt}
			
			We only prove the first limit, the same argument works for the second one.
		   If it is false, by choosing a subsequence,  we assume $\lim_n \alpha_n=\alpha'\neq \alpha$.

			\begin{figure}[h]
				\begin{center}
					\includegraphics[height=3.5cm]{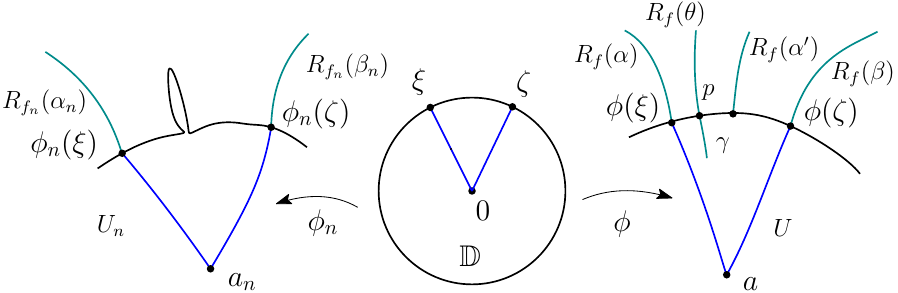}
				\end{center}
				\caption{Rays and convergence}
				\label{fig:e-rays}
			\end{figure}

			Take  $\theta$    lying in between $\alpha$ and $\alpha'$, so that $R_f(\theta)$ lands at a repelling point $p\in \partial U$, see Figure \ref{fig:e-rays} (right).   By Theorem  \ref{cara-riemann}, we have the kernel convergence $(U_{n}, a_n)\rightarrow (U, a)$.  By Lemma \ref{c-internal-ray}, 
			  there exist arcs $\gamma_n: [0,1]\rightarrow \overline{U_{n}}$, $\gamma: [0,1]\rightarrow  \overline{U}$ with the properties  
			
			\begin{itemize}
				\item $\gamma_n([0,1))\subset U_{n}$,  $\gamma_n(1)\in \partial  U_{n}$ is $f_n$-repelling;
				
				\item	$\gamma([0,1))\subset U, \gamma(1)=p$,  $\gamma\cap \phi([0,1]\xi)=\emptyset$;

				\item $\gamma_n\rightarrow\gamma$ in  Hausdorff topology. 
			\end{itemize}

			It follows that  for large $n$, 
			the rays  $\phi_{n}([0,1]\xi)$ and $R_{f_n}(\alpha_n)$  are in different sides of   
			$R_{f_n}(\theta)\cup \gamma_n$.  However, this contradicts the fact that $\phi_{n}([0,1]\xi)$ and $R_{f_n}(\alpha_n)$ have a common endpoint.  The proof of the Claim is completed.


			By the Claim and Lemmas \ref{continuity-rays}, \ref{continuity-rays-general}, we have the Hausdorff  convergence  $$\overline{R_{f_n}(\alpha_n)}\rightarrow \overline{R_{f}(\alpha)}, \ \ \overline{R_{f_n}(\beta_n)}\rightarrow \overline{R_{f}(\beta)}.$$
			
			Let $V_n$ be the  component of $\mathbb C\setminus (\overline{R_{f_n}(\alpha_n)}\cup \overline{R_{f_n}(\beta_n)}\cup \phi_{n}([0,1]\xi)\cup \phi_{n}([0,1]\zeta))$ containing $\phi_{n}(q)$, and let $V$ be the  component of $\mathbb C\setminus (\overline{R_{f}(\alpha)}\cup \overline{R_{f}(\beta)}\cup \phi([0,1]\xi)\cup \phi([0,1]\zeta))$ containing $\phi(q)$. 
			Then $\overline{V_n}\rightarrow \overline{V}$ in Hausdorff topology.
		 If $(q_n)_n$ is a sequence in $\overline{\mathbb D}$  converging  to $q$, any
		accumulation point $b$ of 
			the sequence $(\phi_n(q_n))_n$ is contained $\overline{V}\setminus U$.
			By \cite[Proposition 8.1]{DH}, the set 
			$$\mathcal K:=\{(g, z)\in \mathcal P_d\times \mathbb C; z\in K(g)\}$$
			is closed in  $\mathcal P_d\times \mathbb C$.  Since $(f_n, \phi_n(q_n))\in \mathcal K$, we have $b\in K(f)$. Hence $b\in (\overline{V}\setminus U)\cap K(f)=\sigma_U^{-1}(C_\epsilon)$, where $\sigma_U$ is defined by \eqref{proj-u}.
			
			 Since $\epsilon>0$ is arbitrary,  by the fact  $\bigcap_{\epsilon>0}  \sigma_U^{-1}(C_\epsilon)=L_{U, \phi(q)}$ (see \eqref{intersection-sigma}),
			we conclude that $b\in L_{U, \phi(q)}$.   In particular, if
$\phi(q)\in \partial U$ is not a cut point of $J(f)$ (equivalently $L_{U, \phi(q)}=\{\phi(q)\}$, see Lemma \ref{cont-limb}), we have $b=\phi(q)$.	Hence all convergent subsequences of $(\phi_n(q_n))_n$ have the same limit $\phi(q)$, implying that $\phi_n(q_n)\rightarrow \phi(q)$.
The pointwise convergence follows immediately by taking $(q_n)_n$ to be the constant sequence $(q)_n$. This finishes the proof of (1).

	For (2), note that for large $n$, we have $q_n\in \Gamma_\epsilon$ which implies that $ {R_{f_n}(\theta_n)}\subset V_n$.  Hence $R:= \bigcap_{k\geq 1}\overline{\bigcup_{n\geq k}  {R_{f_n}(\theta_n)}}
			\subset \overline{V}$. Note also $R \cap U=\emptyset$.  It follows that 
			$R \cap K(f)\subset(\overline{V}\setminus U)\cap K(f)=\sigma_U^{-1}(C_\epsilon)$. 
			 Since $\epsilon>0$ is arbitrary, the equality    $\bigcap_{\epsilon>0}  \sigma_U^{-1}(C_\epsilon)=L_{U, \phi(q)}$ implies (2).
			\end{proof}
		


	\section{Exploration of $\overline{\mathcal H_d}$ via $\overline{\mathcal B_d}$} \label{exploration}
		In this section,   we study   $\overline{\mathcal H_d}$ by the algebraic compactification of the space of  Blaschke products.  
	For simplicity, we write $\mathcal B_{d-1,1}$ as $\mathcal B_d$.  Note that $\mathcal B_1$ consists of the identity map.
	
		For each map $f\in \mathcal H_d$, let $\nu_f$  be the landing point of the external ray $R_f(0)$. Clearly $\nu_f$ is continuous in $f\in \mathcal H_d$. Since $\partial U_f(0)$ is a Jordan curve, there is a unique Riemann mapping 
		$\psi_f: U_f(0)\rightarrow \mathbb D$ satisfying that $\psi_f(0)=0, \psi_f(\nu_f)=1$.
		Then $B_f=\psi_f\circ f\circ \psi_f^{-1}$ is a Blaschke product 
	in $\mathcal B_d$.  See the following diagram
		$$
		\xymatrix{ & (U_f(0), 0, \nu_f) 
			\ar[d]_{\psi_f}   \ar[r]^{f}
			& (U_f(0), 0, \nu_f)  \ar[d]^{\psi_f}\\
			&(\mathbb{D}, 0, 1)  \ar[r]_{B_f} & (\mathbb{D}, 0 ,1)}.
		$$ 
		\begin{thm}[Milnor, \cite{Mil12}]\label{homeo-Hd} The map $\Psi:  {\mathcal  H_d} \rightarrow \mathcal B_d$ defined by  $\Psi(f)=B_f$
			is a homeomorphism.
		\end{thm}
	Theorem \ref{homeo-Hd} is a  special case of   \cite[Theorem 5.1]{Mil12}, which gives a canonical  parameterization for all  hyperbolic components in $\mathcal C_d$.
	
	
	The boundary $ \partial  {\mathcal B_d}$ is the disjoint union of the  {\it regular} part $\partial_{\rm reg}  {\mathcal B_d}$ and the {\it  singular} part $\partial_{\rm sing}  {\mathcal B_d}$, defined as
	$$\partial_{\rm reg}  {\mathcal B_d}=\bigsqcup_{2\leq l< d} \Big({\mathcal B_l}\times {\rm Div}_{d-l}(\partial \mathbb D)\Big), \ \ \partial_{\rm sing}  {\mathcal B_d}={\mathcal B_1}\times {\rm Div}_{d-1}(\partial \mathbb D).$$
		We call $D=(B,S)\in \partial  {\mathcal B_d}$
	{\it regular} if $D\in \partial_{\rm reg}  {\mathcal B_d}$;  {\it  singular}   if $D\in \partial_{\rm sing}  {\mathcal B_d}$.
	
	The boundary $\partial \mathcal H_d$ admits a decomposition into the  {\it regular} part $\partial_{\rm reg}  {\mathcal H_d}$ and the {\it  singular} part $\partial_{\rm sing}  {\mathcal H_d}$: 
	$$\partial_{\rm reg}  {\mathcal H_d}=\{f\in \partial \mathcal H_d; |f'(0)|<1\}, \  \partial_{\rm sing}  {\mathcal H_d}=\{f\in \partial \mathcal H_d; |f'(0)|=1\}.$$

			Let $\Phi=\Psi^{-1}:  {\mathcal  B_d} \rightarrow \mathcal H_d$. 
For each $D=(B,S)\in   \partial  {\mathcal B_d}$, define 
$$I_{\Phi}(D)=\Big\{f\in \partial \mathcal H_d;  \text{there exist } (f_n)_n \text{ in  } \mathcal H_d\text{ so that } f_n\rightarrow f \text{ and }  \Psi(f_n)\rightarrow D \Big\}.$$
We call $I_{\Phi}(D)$ the $\Phi$-{\it impression} associated with $D$.
It can be expressed as 
	$$I_{\Phi}(D)=\bigcap_{\epsilon>0}\overline{\Phi(N_{\epsilon}(D))}.$$
	It follows  that $I_{\Phi}(D)$ is a connected and compact subset of $\partial \mathcal H_d$.  It is worth observing that $\partial \mathcal H_d, \partial_{\rm reg} \mathcal H_d,  \partial_{\rm sing} \mathcal H_d$ can be written as
	 $$\partial \mathcal H_d=\bigcup_{D\in  \partial  {\mathcal B_d}}I_{\Phi}(D), \ \partial_{*} \mathcal H_d=\bigcup_{D\in  \partial_*  {\mathcal B_d}}I_{\Phi}(D), \ *\in\{\rm reg, sing\}.$$
	
	In what follows, we focus on the relations between $\partial_{\rm reg}  {\mathcal B_d}$ and $\partial_{\rm reg}  {\mathcal H_d}$, the singular parts  $\partial_{\rm sing}  {\mathcal B_d}$ and $\partial_{\rm sing}  {\mathcal H_d}$  will be discussed in \S \ref{sing-div}.
	
	Let $D=(B,S)\in   \partial_{\rm reg}  {\mathcal B_d}$ and let $f\in I_{\Phi}(D)$.  There exist a sequence $(B_n)_n$ in $\mathcal B_d$     so that
	$B_n\rightarrow D$ and $f_n:=\Phi(B_n)\rightarrow f$.
	 The inverse of the conformal mapping $\psi_{f_n}: (U_{f_n}(0), 0)\rightarrow (\mathbb D, 0)$ is denoted by $\phi_{f_n}: (\mathbb D, 0) \rightarrow (U_{f_n}(0), 0)$.  {By  Lemma \ref{kernel-att}, Theorem \ref{cara-riemann} and Remark \ref{cara-function}, choosing a subsequence if necessary, we assume $\phi_{f_n}$ converges   to a  conformal mapping  $\phi_f: (\mathbb D, 0)\rightarrow (U_f(0), 0)$ in $\mathbb D$.}

The critical points and the zeros  of $f$ outside $U_f(0)$ induce two divisors
$$R_f^0:=\sum_{c\in \mathbb C\setminus U_f(0), f'(c)=0} ({\rm deg}(f,c)-1)\cdot \sigma_{U_f(0)}(c),$$
$$Z_f^0:=\sum_{a\in \mathbb C\setminus U_f(0), f(a)=0} {\rm deg}(f,a)\cdot \sigma_{U_f(0)}(a),$$
where $\sigma_{U_f(0)}: K(f)\setminus U_f(0)\rightarrow \partial U_f(0)$ is defined by \eqref{proj-u} for $U=U_f(0)$.
Note that $R_f^0, Z_f^0 \in {\rm Div}_{{\rm deg}(S)}(\partial U_f(0))$.

Let $h: X\rightarrow Y$ be a homeomorphism between planar sets, let $e\geq 1$ be an integer, the
pull-back $h^*: {\rm Div}_e(Y)\rightarrow {\rm Div}_e(X)$ is defined by 
$$h^*(S)=\sum_{q\in {\rm supp}(S)} \nu(q)\cdot h^{-1}(q), \  \forall \  S=\sum_{q\in {\rm supp}(S)} \nu(q)\cdot q\in  {\rm Div}_e(Y).$$


\begin{pro} \label{div-eq} Let $D=(B, S)\in  \partial_{\rm reg}  {\mathcal B_d}$ and let $f\in I_{\Phi}(D)$.  Assume
	
	(1).  $B_n\rightarrow D,  \  f_n:=\Phi(B_n)\rightarrow f$; 
	
	
	(2).   $\phi_{f_n}$ converges   to $ \phi_f$ in $\mathbb D$.  
	
	Then there is a $\zeta\in \partial \mathbb D$ so that  the   equalities   hold
	$$\zeta B=\phi_{f}^{-1}\circ f\circ \phi_{f}, \  S=\phi_f^*(R_f^0)=\phi_f^*(Z_f^0).$$
	
	In particular, $R_f^0=Z_f^0$.
\end{pro}
We remark that  $\zeta=1$ if $1\notin {\rm supp}(S)$, and $\zeta$  is a number so that  a subsequence  of $(B_n)_n$ converges  to $\zeta B$ in  $\C\setminus {\rm  supp}(S)$ (by 
Lemma \ref{degenerate-0})  if $1\in {\rm supp}(S)$. In the latter case, we shall prove in \S \ref{proof-main} (see Corollary  \ref{unique-zeta})  that $\zeta$ is uniquely determined by  $f$ (not $D$!). 
\begin{proof}
	By Lemma \ref{degenerate-0}, passing to a subsequence if necessary, $B_n$ converges  to $\zeta B$ in $\mathbb D$.  Let  $n\rightarrow \infty$ in the equality 
	$B_n=\phi_{f_n}^{-1}\circ f_n\circ \phi_{f_n}$, we get
	$\zeta B=\phi_{f}^{-1}\circ f\circ \phi_{f}$.

	 Write $S=\sum_{q\in {\rm supp}(S)}\nu(q)\cdot q$.
	Applying Theorem \ref{bet} to the case $(e,m)=(d-1,1)$, for each $q\in {\rm supp}(S)$,    there are exactly $\nu(q)$ critical points of $B_n$ converging  to $q$ as $n\rightarrow \infty$.  By Proposition \ref{kernel-conformal},    the $\phi_{f_n}$-image of these $\nu(q)$  critical points  converge to the $\nu(q)$ critical points of $f$ that are contained in the limb $L_{U_f(0), \phi_f(q)}$.
	The equality $S=\phi_f^*(R_f^0)$ follows immediately.   The same reasoning yields $S=\phi_f^*(Z_f^0)$. Consequently, $R_f^0=Z_f^0$. 
\end{proof}


\begin{pro} \label{pointwise} Let $D=(B, S)\in   \partial_{\rm reg}  {\mathcal B_d}$ and let $f\in I_{\Phi}(D)$. Assume

(a).	$B_n\rightarrow D,  \  f_n:=\Phi(B_n)\rightarrow f$; 
  
  (b). $\phi_{f_n}$ converges  to $ \phi_f$ in $\mathbb D$;
	 
	(c).  $B_n$ converges  to $B_\zeta:=\zeta B$ in  $\C\setminus {\rm  supp}(S)$ for some $\zeta \in \partial \mathbb D$. \footnote{If $1\notin {\rm supp(S)}$,  the condition (c)  is  redundant by Lemma \ref{degenerate-0}. In this case, $\zeta=1$.}
	
 Let $E_\zeta(D)=\bigcup_{j\geq 0} B_\zeta^{-j}({\rm supp}(S))$.

 (1).  If  $q\in \partial{ \mathbb D}\setminus E_\zeta(D)$ is $B_\zeta$-periodic, then $\phi_f(q)$ is an $f$-repelling point.
 
  (2).  If  $q\in  E_\zeta(D)$ is $B_\zeta$-periodic, then $\phi_f(q)$ is an $f$-parabolic point.
 
  (3).  The limb  $L_{U_f(0), \phi_f(q)}$ is trivial if and only if   $q\in \partial{ \mathbb D}\setminus E_\zeta(D)$.
 
  (4). $\phi_{f_n}$ converges pointwisely to $\phi_{f}$ in  $\partial{ \mathbb D}\setminus E_\zeta(D)$.
\end{pro}


\begin{proof} 	Let $q\in \partial{ \mathbb D}$ be a $B_\zeta$-periodic point, then $\phi_f(q)$ is an $f$-periodic point. If  $\phi_f(q)$ is $f$-repelling, then there is only one external ray landing at $\phi_f(q)$ (otherwise, all external rays landing at $f$ persist as we perturb $f$ into a nearby map in $\mathcal H_d$ \cite[Proposition 8.5]{DH}, contradiction!), and the limb  $L_{U_f(0), \phi_f(q)}$ is trivial.   If  $\phi_f(q)$ is $f$-parabolic, then the  limb  $L_{U_f(0), \phi_f(q)}$ is not trivial. In this case, there is an integer $l\geq 0$ so that   $L_{U_f(0), f^l(\phi_f(q))}$ contains a critical point. By the  equalities  $B_\zeta=\phi_{f}^{-1}\circ f\circ \phi_{f}$,  $\phi_f^* R_f^0=S$ given by Proposition \ref{div-eq},  we conclude that in the former case,  $q\in  \partial{ \mathbb D}\setminus E_\zeta(D)$,  while in the latter case, $B_\zeta^l(q)\in {\rm supp}(S)$, which  implies that $q\in E_\zeta(D)$. This proves (1) and (2).

	

	  
	  If    $L_{U_f(0), \phi_f(q)}$ is not  trivial,  by the same reasoning as above, the $f$-orbit of $\phi_f(q)$ meets  either a critical point or a parabolic point. In either case, there is an integer $l\geq 0$ so that the limb $L_{U_f(0), f^l(\phi_f(q))}$ contains a critical point. 
	Again by
 Proposition \ref{div-eq}, we have that $q\in E_\zeta(D)$.   If the limb $L_{U_f(0), \phi_f(q)}$ is   trivial,  Proposition \ref{div-eq} also implies that $q\in \partial{ \mathbb D}\setminus E_\zeta(D)$. This proves (3).
 
	
(4).  It follows from (3) and Proposition \ref{kernel-conformal}.
	\end{proof}

		\begin{pro} \label{unique-rm} Let $D=(B, S)\in \partial_{\rm reg} \mathcal B_{d}$ and $f\in I_{\Phi}(D)$.  There is a unique conformal map $\phi_f:(\mathbb D, 0)\rightarrow (U_f(0),  0)$
			with the  property:
		for any sequence $(f_n)_n$ in $\mathcal H$ converging to $f$,  the conformal maps 
		$(\phi_{f_n})_n$ converge   to $\phi_f$ in $\mathbb D$.
	\end{pro}
	\begin{proof}  
		Let $\phi$ be the limit of a convergent subsequence  $(\phi_{f_{n_k}})_{k}$ of $(\phi_{f_n})_n$.   Then $\phi:(\mathbb D, 0)\rightarrow (U_f(0), 0)$  is conformal. By Proposition \ref{kernel-conformal},  
		$$\limsup_{k\rightarrow \infty} \overline{R_{f_{n_k}}(0)} \bigcap K(f)\subset L_{U_f(0), \phi(1)}.$$
	The fact $\overline{R_f(0)}\subset \limsup_{k\rightarrow \infty} \overline{R_{f_{n_k}}(0)} $ implies that
	 the landing point 	$\nu_f$ of   $R_f(0)$ is in  $L_{U_f(0), \phi(1)}$.
		Hence $\phi(1)$ is the unique $b\in \partial U_f(0)$ whose limb $L_{U_f(0), b}$ contains $\nu_f$. Therefore $\phi$ is uniquely determined by the normalization $\phi(0)=0, \phi(1)=b$.
		
		Since any  convergent subsequence   of $(\phi_{f_n})_n$ has the same limit $\phi$, the   sequence $(\phi_{f_n})_n$   converges to   $\phi$ in $\mathbb D$.
	\end{proof}
	
	\begin{pro} \label{lp-0-ray} Let $D=(B, S)\in  \partial_{\rm reg} \mathcal B_d$ and $f\in I_{\Phi}(D)$.  Let $\nu_f$ be the landing point of the external ray $R_f(0)$, and let $\phi_f$ be given by Proposition \ref{unique-rm}.
		
		(1).  If $1\notin {\rm supp}(S)$, then $\nu_f$ is repelling and  $\nu_f=\phi_f(1)\in \partial U_{f}(0)$.
		
		
		(2).     If $1\in {\rm supp}(S)$, then   $\nu_f\in L_{U_f(0), \phi_f(1)}$. In this case,    either $\nu_f\notin \partial U_{f}(0)$, or  $\nu_f=\phi_f(1)$ and $\nu_f$   is a parabolic fixed point of $f$.
	\end{pro}
\begin{proof} Let $B_n, f_n, \phi_{f_n}, \phi_f, \zeta, E_\zeta(D)$ be given as in Proposition \ref{pointwise}. 
	
	(1).  If $1\notin {\rm supp}(S)$, by Proposition \ref{pointwise}, $\phi_f(1)$ is a repelling fixed point of $f$,  the limb $L_{U_{f}(0), \phi_f(1)}$ is trivial, and $\phi_{f_n}(1)\rightarrow \phi_f(1)$. Note that  point $\nu_{f_n}=\phi_{f_n}(1)$ for each $n$. By the stability of external rays \cite[Proposition 8.1]{DH}, we have $\phi_f(1)=\nu_f$.
	
	(2). If $1\in {\rm supp}(S)$, then $1\in E_\zeta(D)$.  By (the proof of) Proposition  \ref{unique-rm},   $\nu_f\in L_{U_f(0), \phi_f(1)}$.
	If $\nu_f\in \partial U_f(0)$, then $\phi_f(1)=\nu_f$. 
	In this case,  $f(\nu_f)=\nu_f$ implies that   $B_\zeta(1)=1$ (hence $\zeta=1$). By Proposition \ref{pointwise} (2), $\nu_f$ is   a parabolic fixed point of $f$. 
\end{proof}

Let $D=(B, S)\in \partial_{\rm reg} \mathcal B_d$   with $1\notin {\rm supp}(S)$.  For any $f\in I_{\Phi}(D)$, let $\phi_f$ be given by Proposition \ref{unique-rm}. 
Define two maps $\theta_f^{\pm}: \partial \mathbb D\rightarrow \mathbb R/\mathbb Z$ by
$$\theta_f^{\pm}=\theta_{U_{f}(0)}^{\pm}\circ \phi_f|_{\partial \mathbb D},$$
where $\theta_{U_{f}(0)}^{\pm}$ are given in \S \ref{continuity-property}.  By Proposition \ref{pointwise}, they satisfy the  properties: 

\begin{itemize}
	\item if   $q\in \partial{ \mathbb D}\setminus \bigcup_{j\geq 0} B^{-j}({\rm supp}(S))$, then $\theta_f^+(q)=\theta_f^-(q)=\theta$, where $R_{f}(\theta)$ is the unique external ray 
	landing at $\phi_f(q)$.
	
	\item  if   $q\in \bigcup_{j\geq 0} B^{-j}({\rm supp}(S))$, then  $R_{f}(\theta_f^+(q)), R_{f}(\theta_f^-(q))$ land  at $\phi_f(q)$, and the sets  $R_{f}(\theta_f^-(q))$, $L_{U_{f}(0), \phi_f(q)}$, $R_{f}(\theta_f^+(q))$ attach at $ \phi_f(q)$ in positive cyclic order.
\end{itemize}

\begin{pro} \label{com-angles} Let $D=(B, S)\in  \partial_{\rm reg} \mathcal B_d$  with $1\notin {\rm supp}(S)$.  Then for any two maps $f, g\in I_{\Phi}(D)$, we have 
	$$\theta_f^{+}=\theta_g^{+}, \ \theta_f^-=\theta_g^-.$$
In other words, the maps  $\theta_f^{\pm}$
are independent of the choice of $f\in I_{\Phi}(D)$.
\end{pro}

We remark that Proposition \ref{com-angles} is false for  $D=(B, S)\in  \partial_{\rm reg} \mathcal B_d$  with $1\in {\rm supp}(S)$.  In fact there are  maps $f, g\in I_{\Phi}(D)$ with $\theta_f^+\neq \theta_g^+, \ \theta_f^-\neq \theta_g^-$. This fact is not used in this paper, so we omit its proof.

	

\begin{proof}  Since $D=(B, S)$ is regular, the mapping degree $e$ of $B$ satisfies $2 \leq e< d$.
Hence the set $Z=\bigcup_{l\geq 0}B^{-l}(1)$ 
 is dense in $\partial \mathbb D$. To show $\theta^{\pm}_f=\theta^{\pm}_g$, by Lemma  \ref{cont-limb}(2), it suffices to show  $\theta_f^{\pm}|_Z=\theta_g^{\pm}|_Z$.   In the following, we shall determine the precise value of $\theta_f^{\pm}$ on $Z$. 
 
Set $Z_0=\{1\}$, $Z_l=B^{-l}(1)\setminus B^{-(l-1)}(1)$ for $l\geq 1$, then $Z=\bigsqcup_{l\geq 0}Z_l$.
For $q, q'\in \partial\mathbb D$, let $[q,q']\subset \partial\mathbb D$ be an (closed) arc segment on $\partial\mathbb D$ with endpoints $q,q'$ so that $q, \zeta, q'$ are in the counter-clockwise order, for any $\zeta \in [q,q']\setminus\{q,q'\}$.
Let 
 $[q, q')=[q,q']\setminus\{q'\}$.
Note that the divisor $S=\sum_{q\in {\rm supp}(S)}\nu(q)\cdot q$ 
  induces a function $\nu: \partial \mathbb D\rightarrow \mathbb N$ so that $\nu(q)>0$ if and only if $q\in {\rm supp}(S)$.

By Propositions \ref{pointwise} (2) and  \ref{lp-0-ray}(1),  $\theta_f^+(1)=\theta_f^-(1)=0$.
 To determine $\theta_f^-|_{Z_1}$, write the points in $B^{-1}(1)$ as $q_0=1,  q_1, \cdots, q_{e-1}, q_e=q_0$,    in  the counter-clockwise order on $\partial \mathbb D$.
 For any $1\leq j<e$,  by   the divisor equality  $\phi_f^* R_f^0=S$ given by Proposition \ref{div-eq} and the relation between the angular width and the number of critical points  \cite[\S 2]{GM93},  we have that
 $$\theta_f^-(q_j)-\theta_f^-(q_0)=\frac{2\pi }{d}\Big(j+\sum_{\zeta\in [q_0, q_j)\cap {\rm supp}(S)} \nu(\zeta)\Big),$$
 $$\theta_f^+(q_j)-\theta_f^-(q_j)=\frac{2\pi \nu(q_j)}{d}.$$
 In this way, $\theta_f^+$ and $\theta_f^-$ are determined on $Z_1$.
 Assume by induction that  $\theta_f^+$ and $\theta_f^-$  are determined in $B^{-k}(1)$ for some $k\geq 1$. Take two adjacent points  {$p, p'\in B^{-k}(1)$ so that $[p, p']$ is disjoint from $ B^{-k}(1)\setminus \{p,p'\}$}. Then $B^{-k-1}(1)\cap [p, p']$ consists of $e+1$ points, labeled in   the counter-clockwise order as $q_0=p, q_1, \cdots, q_e=p'$.
  For any $1\leq j<e$, again by Proposition \ref{div-eq} and  \cite[\S 2]{GM93}, 
 $$\theta_f^-(q_j)-\theta_f^-(q_0)=\frac{2\pi }{d^{k+1}}\Big(j+\sum_{\zeta\in [q_0, q_j)\cap {\rm supp}(S)} \nu(\zeta)\Big),$$
 $$\theta_f^+(q_j)-\theta_f^-(q_j)=\frac{2\pi \nu(q_j)}{d^{k+1}}.$$
  In this way, $\theta_f^+$ and $\theta_f^-$ are determined on $Z_{k+1}$. By induction,   $\theta_f^{\pm}$  are determined on $Z$. 
  
  Note that $\theta_g^{\pm}|_{Z}$ are  determined in the same fashion, we  get $\theta_f^{\pm}|_Z=\theta_g^{\pm}|_Z$. 
	\end{proof}

	
	
	

	\section{Regular divisors:    singleton case}\label{regular}
	
	In this section, we show
	
	\begin{pro}\label{imp-singleton}   Let $D=(B,S)\in \partial_{\rm reg} \mathcal B_{d}$.  If
	$1\notin {\rm supp}(S)$,  $S$ is simple,  and $D$ has no dynamical relation, then $I_{\Phi}(D)$ is   a singleton.
	\end{pro}

Here   $D=(B,S)\in \partial_{\rm reg}  {\mathcal B_d}$ has \textbf{dynamical relation} means that there are different points $q,q'\in {\rm supp}(S)$ and an integer $l\geq 1$
so that $B^l(q)=q'$.


We need the following theorems, established in the prequel \cite{CWY}:

\begin{thm}  [\cite{CWY}, Theorem 1]  \label{lc-J} For any $f\in \partial_{\rm reg}\mathcal H$, the Julia set $J(f)$ is locally connected.
	\end{thm}

\begin{thm}  [\cite{CWY}, Theorem 2]\label{rigidity} If $f,g\in \partial_{\rm reg}\mathcal H$ are topologically conjugate $\phi\circ f=g\circ \phi$ by a homeomorphism $\phi:\mathbb C\rightarrow \mathbb C$, which is conformal in the Fatou set $F(f)$ with normalization $\phi'(\infty)=1$, then $f=g$.
\end{thm}

\begin{lem}  \label{char-map-imp}   Let $D=(B,S)\in \partial \mathcal B_{d}$  satisfy the condition of Proposition \ref{imp-singleton}.  
	For any $f\in I_{\Phi}(D)$,  let $\phi_f$ be given by Proposition \ref{unique-rm}.
	
	(1). If $q\in {\rm supp}(S)$ is $B$-periodic, then $\phi_f(q)$ is  a parabolic point of $f$. Further, let $l$ be the $B$-period of $q$, 
	then
	$(f^{l})'(\phi_f(q))=1$ and there is precisely one   parabolic Fatou component whose boundary contains  $\phi_f(q)$.
	
	(2). If $q\in {\rm supp}(S)$ is not $B$-periodic, then $\phi_f(q)$ is a critical point of $f$.
	\end{lem}
\begin{proof}  (1). It follows from Proposition \ref{pointwise} that $\phi_f(q)$ is  a parabolic point of $f$. 
	Note that  $l$ equals  the $f$-period  of $\phi_f(q)$. By Theorem \ref{RY}, there is an $f^l$-invariant external  ray  $R_f(\theta)=f^l(R_f(\theta))$ landing at $\phi_f(q)$. By the Snail Lemma \cite[Lemma 16.2]{Mil06}, $(f^{l})'(\phi_f(q))=1$.
	By the assumption on $D$ and the divisor equality  $\phi_f^* R_f^0=S$ proven by Proposition \ref{div-eq},  there is only one critical point in $L_{U_f(0), \phi_f(q)}\cup\cdots\cup L_{U_f(0), f^{l-1}(\phi_f(q))}$.
	Since each cycle of   parabolic Fatou component   contains  at least one critical point, we conclude that there is precisely  one   parabolic Fatou component whose boundary contains  $\phi_f(q)$.
	
	(2).  By the assumption that $D$ has no dynamical relation,   and the divisor equality  $\phi_f^* R_f^0=S$ proven by Proposition \ref{div-eq},  the limb $L_{U_f(0), f^k(\phi_f(q))}$ contains no critical point for all $k\geq 1$.  By Theorem \ref{RY}, $L_{U_f(0), f(\phi_f(q))}$ is  trivial.  On the other hand, the limb $L_{U_f(0), \phi_f(q)}$ is not trivial, implying that $\phi_f(q)$ is a critical point of $f$.
	\end{proof}

 Here is a supplement to Lemma \ref{char-map-imp}.  Let $Q$ (possibly empty) consist of all $B$-periodic points $q\in {\rm supp}(S)$. For each $q\in Q$, let $P_f(q)$ be   parabolic Fatou component  whose boundary contains  $\phi_f(q)$.  Then all critical points of $f$ are  contained in
 $$\overline{U_f(0)}\cup \bigcup_{q\in Q} P_f(q).$$

In the following, let $D=(B,S)\in \partial \mathcal B_{d}$  satisfy the condition of Proposition \ref{imp-singleton}. 
For any $f\in I_{\Phi}(D)$,  let $\phi_f$ be given by Proposition \ref{unique-rm}.
Let $$X_0(f)=\overline{U_f(0)}\cup \bigcup_{l\in \mathbb N} \bigcup_{q\in Q} f^l(\overline{P_f(q)}).$$
Clearly $f(X_0(f))=X_0(f)$. For any $n\in\mathbb{N}$, define inductively $X_{n+1}(f)$ to be the connected component of $f^{-1}(X_n(f))$ containing $X_n(f)$.  Then we have an increasing sequence of connected and compact sets $$X_0(f)\subset X_1(f)\subset X_2(f)\subset\cdots.$$ Each $X_n(f)$ is a finite union of closed disks, of which any two are either disjoint or touching at exactly one point on the boundaries.  

Let $$Y(f)=\overline{\bigcup_{n\in\mathbb{N}} X_n(f)}, \ Y_\infty(f)= Y(f)\setminus\bigcup_{n\in\mathbb{N}} X_n(f).$$
Note that $Y_\infty(f)$ is  the set of all limit points on $Y(f)$. 

The following fact describes the structure of the filled Julia set $K(f)$. It  is a special case of \cite[Theorem 1.3]{CWY}. 


\begin{pro}  [\cite{CWY}, Theorem 1.3] \label{limit-pt-one-ray}  The filled Julia set $K(f)=Y(f)$. Further, for each $x\in Y_\infty(f)$, there is exactly one external ray landing at $x$.
	\end{pro}

By the local connectivity of $J(f)$ (see Theorem \ref{lc-J}), for each $\theta\in \mathbb R/\mathbb Z$, the external ray $R_f(\theta)$ lands at a point $b_f(\theta)\in J(f)$. 
The \emph{real lamination} $\lambda_\mathbb{R}(f)\subset(\mathbb R / \mathbb Z)^2$ of $f$  consists of   $(\theta_1,\theta_2)\in (\mathbb R / \mathbb Z)^2$ for which $b_f(\theta_1)=b_f(\theta_2)$.

	Define $\tau:  \mathbb R/\mathbb Z\rightarrow  \mathbb R/\mathbb Z$  by $ t\mapsto dt$.

	

\begin{lem} \label{equal-lam}The real lamination $\lambda_\mathbb{R}(f)$ is independent of  $f\in  I_{\Phi}(D)$. In other words, for any $f,g\in I_{\Phi}(D)$, we have $\lambda_\mathbb{R}(f)=\lambda_\mathbb{R}(g)$.
	\end{lem}
\begin{proof} 
	
Take $(\alpha, \beta)\in \lambda_\mathbb{R}(f)$ with $\alpha\neq \beta$. 
	We claim that the orbit $b_f(\alpha)\mapsto f(b_f(\alpha))\mapsto\cdots$ meets either a critical point or  a  parabolic point on $\partial U_f(0)$.
	
	By Proposition \ref{limit-pt-one-ray},  $b_f(\alpha)\notin Y_\infty(f)$.  Hence $b_f(\alpha)\in \bigcup_{n\in\mathbb{N}} \partial X_n(f)$, and it is a intersection point of the boundaries of two adjacent Fatou components. By the construction of $X_n(f)$, there is a minimal integer $l\geq 0$ so that $w=f^l(b_f(\alpha))\in  \partial U_{f}(0)$, and at least two external rays land at $w$. By Theorem \ref{RY}, 
  the $f$-orbit of $w$  meets either a   critical point or  a  parabolic point. 
	
	By the claim, there is an integer $m\geq l$ and $q\in {\rm supp}(S)$, so that 
	$$(d^m \alpha, d^m\beta)=(\theta_f^+(q), \theta_f^-(q)) \ \text{ or } (\theta_f^-(q), \theta_f^+(q)).$$
	
	We may assume $(d^m \alpha, d^m\beta)=(\theta_f^-(q), \theta_f^+(q))$.
	In the following, we shall show   that  $\beta$ is uniquely determined by  $\alpha$.  
	
Let 
$(\alpha',\beta')\in \lambda_{\mathbb R}(f)\cap ( \tau^{-1}(\theta_f^-(q))\times \tau^{-1}(\theta_f^+(q))$. Clearly $b_f(\alpha')\in f^{-1}(\phi_f(q))$. Note that  $f^{-1}(\phi_f(q))$  consists of $d$   points,  distributed in $d$ different limbs: 
	$$L_{U_f(0), \phi_f(q')}, \ L_{U_f(0), \phi_f(q'')}, \ \ \text{ where } q'\in B^{-1}(q), q''\in {\rm supp}(S).$$ 
	If $b_f(\alpha')\in L_{U_f(0), \phi_f(q')}$ for some $q'\in B^{-1}(q)$, then $b_f(\alpha')=\phi_f(q')$. In this case $\alpha'=\theta_f^-(q')$ and $\beta'=\theta_f^+(q')$.  If $b_f(\alpha')\in L_{U_f(0), \phi_f(q'')}$ for some $q''\in {\rm supp}(S)$, then $b_f(\alpha')\neq \phi_f(q'')$ since $D$ has no dynamical relation. In this case,  there is a unique external ray   $R_f(t_0)$ with $t_0\in \tau^{-1}(\theta_f^+(q))$ landing at  $b_f(\alpha')$,
	and $\beta'=t_0$. It follows that in either case,  $\beta'$ is uniquely determined once $\alpha'$ is given. 
	
	By the same reasoning and induction, $\beta\in \tau^{-m}(\theta_f^+(q))$ is  uniquely determined under the condition  $(\alpha, \beta)\in \lambda_\mathbb{R}(f)$  once $\alpha\in  \tau^{-m}(\theta_f^-(q))$ is given. 
	
	By Proposition \ref{com-angles}, $\lambda_\mathbb{R}(f)$ is uniquely determined and is independent of  $f\in  I_{\Phi}(D)$.
	\end{proof}

\begin{proof}[Proof of Proposition \ref{imp-singleton}]  Let  $f, g\in  I_{\Phi}(D)$. The idea is to construct a topological conjugacy $h$ between $f$ and $g$, and then apply  rigidity (Theorem \ref{rigidity}).   In the proof, let $*$ denote the map $f$ or $g$.


Let $\psi_{*, \infty}: \mathbb C\setminus K(*)\rightarrow \mathbb C\setminus \overline{\mathbb D}$ be the B\"ottcher map of $*$, normalized so that 
$\psi_{*, \infty}(z)=z+O(1)$ near $\infty$. Then $h_\infty=\psi_{g, \infty}^{-1}\circ \psi_{f, \infty}:  \mathbb C\setminus K(f)\rightarrow  \mathbb C\setminus K(g)$ is a conformal conjugacy: $h_\infty\circ f=g\circ h_\infty$.

By Lemma \ref{equal-lam}, $ \lambda_{\mathbb R}(f)=\lambda_{\mathbb R}(g)$. Hence $h_\infty$ extends to a homeomorphism $h_\infty:  (\mathbb C\setminus K(f)) \cup J(f)\rightarrow  (\mathbb C\setminus K(g))\cup J(g)$ by defining $h_\infty(b_{f}(\theta))=b_{g}(\theta)$
for all $\theta \in \mathbb R/\mathbb Z$.  It keeps the conjugacy $h_\infty\circ f|_{J(f)}=g\circ h_\infty|_{J(f)}$.\vspace{3pt}

In the following, we need to define the conjugacy piece by piece  in each bounded  Fatou component of $f$,  and then glue them together.  

 \vspace{5pt}
\textbf{Conjugacy in $U_f(0)$.}  By Proposition \ref{div-eq},  $B=\phi_{f}^{-1}\circ f\circ \phi_{f}=\phi_{g}^{-1}\circ g\circ \phi_{g}$. The map $h_0=\phi_g\circ \phi_f^{-1}: U_f(0)\rightarrow U_g(0)$ is a conformal conjugacy between $f|_{U_f(0)}$ and $g|_{U_g(0)}$. 
This $h_0$ extends to a homeomorphsim  $h_0: \overline{U_f(0)}\rightarrow \overline{U_g(0)}$. By Proposition \ref{lp-0-ray}, $\phi_f(1)=b_f(0)$, $\phi_g(1)=b_g(0)$. Note that both $h_0|_{\partial U_f(0)}$ and $h_\infty|_{\partial U_f(0)}$ are orientation preserving,  satisfying that
$$\phi \circ f|_{\partial U_f(0)}= g|_{\partial U_g(0)}\circ \phi , \ \phi(b_f(0))=b_g(0), \  \ \phi\in \{h_0|_{\partial U_f(0)}, h_\infty|_{\partial U_f(0)}\},$$
we conclude that $h_0|_{\partial U_f(0)}=h_\infty|_{\partial U_f(0)}$.
 
  \vspace{5pt}
\textbf{Conjugacy in parabolic basins.}   For  each $q\in Q$, let $l_q$ be the $B$-period of $q$.   
 Let $c_*(q)$ be the unique $*$-critical point in the parabolic Fatou component $P_*(q)$.  There is a unique conformal map 
 $\phi_{*, q}: P_*(q) \rightarrow \mathbb D$ with  $\phi_{*, q}(c_*(q))=0$ and $\phi_{*, q}(\phi_*(q))=1$. Note that $B_{*,q}=\phi_{*, q}\circ *^{l_q}|_{P_*(q)}\circ \phi_{*, q}^{-1}$ is a degree two Blaschke product,  with  a parabolic fixed point  at $1$ of multiplicity $3$ and  a critical point at $0$.  This map takes the form (see \cite[\S 6]{Mc88}):   
 $$B_{*,q}(z)=\frac{3z^2+1}{z^2+3}.$$
 It follows that $h_{q,0}=\phi_{g, q}^{-1}\circ \phi_{f, q}: P_f(q) \rightarrow P_g(q)$ is a conformal  conjugacy between $f^{l_q}|_{P_f(q)}$ and $g^{l_q}|_{P_g(q)}$. 
 For each $1\leq k<l_q$,  set   
 $$h_{q,k}= g^{l_q-k}|_{g^k(P_g(q))}\circ h_{q,0}\circ f^{l_q-k}|^{-1}_{f^k(P_f(q))}: f^k(P_f(q)) \rightarrow g^k(P_g(q)).$$
 The maps $(h_{q,k})_{0\leq k<l_q}$ can  extend to   homeomorphisms between the closures of the domains and ranges.
 
 Note that  $h_{q,0}|_{\partial P_f(q)}$ and $h_\infty|_{\partial P_f(q)}$ are orientation preserving,    satisfying 
 $$\phi \circ f^{l_q}|_{\partial P_f(q)}= g^{l_q}|_{\partial P_g(q)}\circ \phi, \  \ \phi\in \{ h_{q,0}|_{\partial P_f(q)},  h_\infty|_{\partial P_f(q)}\}.$$
It is worth noting that $h_{q,0}(\phi_f(q))=\phi_g(q)$, $h_{\infty}(b_f(\theta_f^{+}(q)))=b_g(\theta_f^{+}(q))$ and
 $\phi_f(q)=b_f(\theta_f^{+}(q))$, $\phi_g(q)=b_g(\theta_g^{+}(q))$. By Proposition \ref{com-angles}, $\theta_f^{+}(q)=\theta_g^{+}(q)$. Hence $h_{q,0}|_{\partial P_f(q)}$ and $h_\infty|_{\partial P_f(q)}$ have the same normalization. Consequently, $h_{q,0}|_{\partial P_f(q)}=h_\infty|_{\partial P_f(q)}$. 
 
  Similarly, $h_{q,k}|_{\partial f^k(P_f(q))}=h_\infty|_{\partial  f^k(P_f(q))}$ for each $1\leq k<l_q$.   
 
 \vspace{5pt}
 \textbf{Conjugacy in aperiodic Fatou components.}  
 Let 
 $$A_0=\{\theta\in \tau^{-1}(0); b_f(\theta)\notin \partial U_f(0)\}, \ \Theta=\bigcup_{k\geq 0} \tau^{-k}(A_0).$$
Let  $\mathcal F_*(0)$ consist of all components of $\bigcup_{k\geq 0}*^{-k}(U_*(0))$ other than $U_*(0)$.
 Note that for each $\theta\in \Theta$, there is a unique $V_*^\theta\in \mathcal F_*(0)$ so that $b_*(\theta)\in \partial V_*^\theta$, and vice visa. Hence  there is  a bijection  between $\Theta$   and $\mathcal F_*(0)$.
 
 For each $q\in Q$, 
  let 
 \bess
 A_q&=&\big\{\theta_f^+(B^k(q)); 0\leq k<l_q\big \},\\
 A'_q&=&\Big\{\theta\in \tau^{-1}(A_q); \  b_f(\theta)\notin \bigcup_{0\leq k<l_q} \partial f^k(P_f(q))\Big\},\\
 \Theta_q&=&\bigcup_{k\geq 0} \tau^{-k}(A'_q).
 \eess
 Let $\mathcal F_*(q)$ be the collection of aperiodic components of $\bigcup_{k\geq 0}*^{-k}(P_*(q))$.
One may verify that  for each $\theta\in \Theta_q$, there is a unique $V_*^\theta\in \mathcal F_*(q)$ so that $b_*(\theta)\in \partial V_*^\theta$, and vice visa. Hence  there is  a bijection  between $\Theta_q$   and $\mathcal F_*(q)$.

For each $V\in \mathcal F_f(0)$, write $V=V_f^\theta$ for some $\theta\in  \Theta$. Let $l\geq 1$ be minimal so that $\tau^l(\theta)=0$.  Define $h_V:  V_f^\theta\rightarrow V_g^\theta$ by $h_V= g^l|^{-1}_{V_g^\theta}\circ h_0\circ f^l|_{V_f^\theta}$. This $h_V$ extends to the boundary $\partial V$ and satisfies  $h_V|_{\partial V}=h_\infty|_{\partial V}$.

Similarly, for each $V\in \mathcal F_f(q)$ with $q\in Q$, write $V=V_f^\theta$ for some $\theta\in  \Theta_q$. Let $l\geq 1$ be minimal so that $\tau^l(\theta)\in A_q$.  Assume $\tau^l(\theta)=\theta_f^+(B^k(q))$ for some $0\leq k<l_q$. Define $h_V:  V_f^\theta\rightarrow V_g^\theta$ by $h_V= g^l|^{-1}_{V_g^\theta}\circ h_{q,k}\circ f^l|_{V_f^\theta}$. This $h_V$ extends to the boundary and satisfies $h_V|_{\partial V}=h_\infty|_{\partial V}$.


\vspace{5pt}
\textbf{Gluing maps and applying rigidity.}  By gluing the maps  in 
$$\Big\{h_{\infty}, h_{0}, h_{q,k},  h_V;  \ q\in Q, 0\leq k<l_q, \  V\in \mathcal F_f(0)\cup \bigcup_{q\in Q} \mathcal F_f(q)\Big\},$$ 
we get a homeomorphism $h: \mathbb C\rightarrow \mathbb C$.
It is a topological conjugacy between $f$ and $g$, conformal in  Fatou set $F(f)$,   normalized as $h'(\infty)=1$.   By Theorem \ref{rigidity}, $f=g$.  Hence  $I_{\Phi}(D)$ is a singleton.
	\end{proof}

	\section{Regular divisors:  non singleton case} \label{regular-nss}
	In this section, we show
	
	\begin{pro}  \label{regular-ns}  Let $D=(B,S)\in \partial_{\rm reg} \mathcal B_{d}$, then $I_{\Phi}(D)$ is not a singleton in either of the following situations:
		
		(1).  $1\notin {\rm supp}(S)$ and $S$ is not simple.
		
		(2).  $1\notin {\rm supp}(S)$,  $S$ is simple, and $D$ has dynamical relation.
		
		(3). 		$1\in {\rm supp}(S)$.
	\end{pro}

	 We need some lemmas.  
	 
	 Given a closed curve $\gamma: [0,1]\rightarrow \mathbb C$ and a point $a\notin \gamma$, there is  a parameterization  $\gamma(t)-a=\rho(t)e^{i\theta(t)}, t\in [0,1]$, where $\rho, \theta$ are continuous. The {\it winding number} $w(\gamma, a)$ is defined as $(\theta(1)-\theta(0))/2\pi$.
	 The following fact is standard.
			\begin{lem} \label{winding-intermediate} Let $U$ be a Jordan disk in $\mathbb C$, and let $h:\overline{U}\rightarrow \mathbb C$ be   continuous. 
				If   $w(h(\partial U), a)\neq 0$ for some $a\notin h(\partial U)$, then there is  $p\in U$ with
				$h(p)=a$.
		\end{lem}
	\begin{proof} If not, then $a\notin h(\overline{U})$. Note that $\partial U$ is homotopic in $\overline{U}$ to a constant curve $\gamma_0\subset \overline{U}$.  It follows that $h(\partial U)$ is homotopic to $h(\gamma_0)$ in $\mathbb C\setminus \{a\}$.
		Since the winding number is a homotopy invariant,  we have $w(h(\partial U), a)=w(h(\gamma_0), a)= 0$.
		This is a contradiction.
		\end{proof}
	
	 Let $D=(B,S)\in \partial_{\rm reg} \mathcal B_{d}$.
 Suppose $S$ is simple. 
When $\widehat{B}\in \mathcal B_d$ is sufficiently close to $D$, for each  $q\in {\rm supp}(S)$, 
there is a unique zero of  $\widehat{B}$ that is close to $q$, denote this zero by $z_q(\widehat{B})$;  by Theorem \ref{bet}, there is also a  unique critical point of  $\widehat{B}$   close to $q$,  denote  this  critical point by $c_q(\widehat{B})$.

\begin{lem} \label{cont-orbit} 	 Let $D=(B,S)\in \partial_{\rm reg} \mathcal B_{d}$. Suppose $S$ is simple and $1\notin {\rm supp}(S)$. Let $q\in {\rm supp}(S)$  and let  $l\geq 1$ be an integer so that $\{ B^{k}(q); 1\leq k<l \}\cap {\rm supp}(S)=\emptyset$.
Then for any sequence $(B_n)_n\subset \mathcal B_d$ converging to $D$ algebraically, we have 
	$$B_n^l(c_{q}(B_n))\rightarrow B^{l}(q).$$
	\end{lem}
	\begin{proof}  We first claim  $B_n(c_{q}(B_n))\rightarrow B(q)$.  If it is not true, by passing to  subsequence, we assume     $B_n(c_q(B_n))\notin \mathbb D(B(q), \delta)$ for some $\delta>0$ and for all $n$.
		Choose small $r>0$ so that $\overline{B(\mathbb D(q, r))}\subset  \mathbb D(B(q), \delta)$. Let $A$ be a thin annular neighborhood of $\partial \mathbb D(q, r)$ so that $B|_A$ is univalent and  $B(A)\subset \mathbb D(B(q), \delta)$. 
		By Lemma \ref{degenerate-0},      $B_n$ converges 
		uniformly to $B$ in $A$. By Rouche's Theorem, $B_n|_{\partial \mathbb D(q, r)}$ is injective for large $n$, 
		 hence $B_n(\partial \mathbb D(q, r))$ is a Jordan curve in $\mathbb D(B(q), \delta)$.
		Let $V_n$ be the  component of $\C \setminus B_n(\partial \mathbb D(q, r))$   containing $B(q)$. Let $U_n$ be the component of $B_n^{-1}(V_n)$ such that $\partial \mathbb D(q, r)\subset \partial U_n$ and $U_n\subset  \mathbb D(q, r)$.
		The assumption $B_n(c_q(B_n))\notin \mathbb D(B(q), \delta)$ implies that 
		$U_n$ contains no critical point of $B_n$. Hence $B_n: U_n\rightarrow V_n$ is conformal, which implies that $U_n$ is simply connected.   Therefore $U_n=\mathbb D(q, r)$. However this is a contradiction since $c_q(B_n)\in \mathbb D(q, r)$.
		
		By the claim and Lemma \ref{degenerate-0},  along with the assumption that $\{ B^{k}(q); 1\leq k<l \}\cap {\rm supp}(S)=\emptyset$,  we conclude that $B_n^l(c_{q}(B_n))\rightarrow B^{l}(q)$.
		\end{proof}

\begin{pro} \label{prescribed-hyp-d} 	 Let $D=(B,S)\in \partial_{\rm reg} \mathcal B_{d}$. Suppose $S$ is simple and $1\notin {\rm supp}(S)$.
	Let $q\in {\rm supp}(S)$,  and let  $l\geq 1$ be an integer so that $\{ B^{k}(q); 1\leq k<l \}\cap {\rm supp}(S)=\emptyset$ and $q':= B^{l}(q)\in {\rm supp}(S)$.
	 Then for any $L\geq 0$ and any small   $\epsilon>0$, there is $\widehat{B}\in \mathcal B_d\cap N_{\epsilon}(D)$ such that the hyperbolic distance 
	$$d_{\mathbb D}(z_{q'}(\widehat{B}), \widehat{B}^l(c_{q}(\widehat{B}))=L.$$
	\end{pro} 
\begin{proof}  Fix $L\geq 0$ and $\epsilon>0$. 
	For  $\delta\in (0, \epsilon)$, let $\alpha=\mathbb D\cap \partial\mathbb D(q, \delta)$ and $\beta=\partial\mathbb D\cap \mathbb D(q, \delta)$ be circular arcs, with common endpoints $a,b\in \partial \mathbb D$.  We may assume $\delta$ is small  so that $1\notin \overline{\beta}$ and $B^k(\overline{\beta})\cap ({\rm supp}(S)\cup \overline{\beta})=\emptyset$ for all $1\leq k<l$.
	
	For each $\zeta\in \alpha$, let $B_\zeta\in  \mathcal B_{{\rm deg}(B)+1}$ be determined by the divisor equality $Z(B_\zeta)=Z(B)+1\cdot \zeta$. Then we 
	get a continuous map
$$\gamma: \alpha \rightarrow \mathbb D, \ \zeta\mapsto B_{\zeta}^l(c_{q}(B_\zeta)).$$
Note that as $\zeta$ approaches  $\omega\in \{a,b\}$ along $\alpha$, $B_\zeta$ converges to $(B, 1\cdot \omega)$ algebraically. 
By Lemma \ref{cont-orbit}, we have $\gamma(a)=B^{l}(a)$ and $\gamma(b)=B^{l}(b)$.


Let $\tau\in (0, \epsilon)$.  Note that  each multipoint $$\mathbf{x}=(x_p)_{p\in {\rm supp}(S)\setminus\{q\}} \in X_\tau:=\prod_{p\in {\rm supp}(S)\setminus\{q\}} (\mathbb D(p, \tau)\cap \mathbb D), $$ 
 induces a divisor $D_\mathbf{x}=\sum_{p\in {\rm supp}(S)\setminus\{q\}} 1\cdot x_p \in {\rm Div}_{{\rm deg}(S)-1}(\mathbb D)$.  For any $(\zeta, \mathbf{x})\in \alpha\times X_{\tau}$, 
there are Blaschke products $\widehat{B}_{\mathbf x}\in \mathcal B_{d-1}$,  $\widehat{B}_{\zeta, \mathbf x}\in \mathcal B_d$  determined by the  divisor equalities 
\begin{equation}\label{zero-div-eq}
	Z(\widehat{B}_{\mathbf x})=Z(B)+D_\mathbf{x}, \ 
	Z(\widehat{B}_{\zeta, \mathbf x})=Z(B)+ 1\cdot \zeta+D_\mathbf{x}.
\end{equation}
 We may  assume $\tau$ is small so that
\begin{itemize}
	\item    $\bigcup_{1\leq k<l}\widehat{B}_{\mathbf x}^k(\overline{\beta})$ is disjoint from the  $\epsilon_0$-neighborhood of    $\overline{\beta}\cup \bigcup_{p\in {\rm supp}(S)\setminus\{q\}}\overline{\mathbb D(p, \tau)}$,  for some $\epsilon_0>0$ and for all  $\mathbf x\in X_{\tau}$;
	
	\item  $\overline{\mathbb D_{\rm hyp}(x_{q'}, L)}\subset \mathbb D(q', \min\{ \epsilon,   d(\gamma, q')/2\})$  for all  $x_{q'}\in \mathbb D(q', \tau)\cap\mathbb D$, where  $ d(\gamma, q')=\min_{w\in \gamma}|q'-w|$. 
\end{itemize}

The map  $$H:  
\begin{cases} \alpha\times X_{\tau} \rightarrow  \mathbb D, \\
	(\zeta, \mathbf x)\mapsto \widehat{B}^l_{\zeta, \mathbf x}(c_{q}(\widehat{B}_{\zeta, \mathbf x})) 
\end{cases}$$
is continuous. By   Lemmas \ref{degenerate-0} and  \ref{cont-orbit}, $H$ extends to a continuous map $\overline{H}: \overline{\alpha}\times \overline{X_{\tau}}\rightarrow \overline{\mathbb D}$. Since $ \overline{\alpha}\times \overline{X_{\tau}}$ is compact, $\overline{H}$ is uniformly continuous.  It follows that there is $\tau'<\tau$ so that 
\begin{equation} \label{H-close}
	|\overline{H}(\zeta, \mathbf x)-\overline{H}(\zeta, \mathbf x_0)|< d(\gamma, q')/2,  \  \forall \ (\zeta, \mathbf x)\in \overline{\alpha}\times \overline{X_{\tau'}},
	\end{equation}
  where
$\mathbf{x_0}=(p)_{p\in {\rm supp}(S)\setminus\{q\}}$. 
Note that $\overline{H}(\zeta, \mathbf x_0)=B_{\zeta}^l(c_{q}(B_\zeta))$ for $\zeta\in \alpha$, and $\gamma=\overline{H}(\alpha, \mathbf x_0)$.



 	
 

\begin{figure}[h]
	\begin{center}
		\includegraphics[height=5cm]{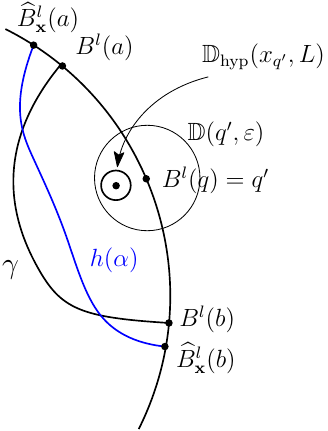}
	\end{center}
\caption{Finding Blaschke product with prescribed mapping behavior.}
\label{fig: prescribeB}
\end{figure}

In the following, we fix some $\mathbf x \in  {X_{\tau'}}$. 
Let  $U:=\mathbb D\cap \mathbb D(q, \delta)$.  For each $\zeta \in U$,  the equation \eqref{zero-div-eq}  determines a unique   $\widehat{B}_{\zeta, \mathbf x}\in  \mathcal B_{d}$.
The map $$h:  
\begin{cases} U \rightarrow  \mathbb C, \\
	\zeta\mapsto \widehat{B}^l_ {\zeta, \mathbf x}(c_{q}(\widehat{B}_{\zeta, \mathbf x})) 
\end{cases}$$
  is continuous.  As $\zeta$ approaches  $s\in \partial \mathbb D\cap \partial U=\overline{\beta}$,  $\widehat{B}_{\zeta, \mathbf x}$ converges to $(\widehat{B}_{\mathbf x}, 1\cdot s)$ algebraically.
   By the assumptions  $1\notin \overline{\beta}$,  $\widehat{B}_{\mathbf x}^k(\overline{\beta})\cap   \overline{\beta}=\emptyset$ for all $1\leq k<l$ and  Lemma \ref{cont-orbit},  { $h(\zeta)$ tends to $\widehat{B}_{\mathbf x}^l(s)\in \partial \mathbb D$. } By defining $h(s)=\widehat{B}_{\mathbf x}^l(s)$ for $s\in \overline{\beta}$,  we can extend $h$  to a continuous map $h: \overline{U}\rightarrow \mathbb C$. 
Note that    $h(\partial U)$ has two parts:  $h(\alpha)=\overline{H}(\alpha, \mathbf x)$ which is $d(\gamma, q')/2$-close to $\gamma$ by \eqref{H-close}, and  $h(\overline{\beta})=\widehat{B}_{\mathbf x}^l( \overline{\beta})\subset \partial\mathbb D$, see Figure \ref{fig: prescribeB}.

  Take an arbitrary point $\xi\in \partial \mathbb D_{\rm hyp}(x_{q'}, L)$, then $\xi\notin h(\partial U)$ and the winding number $w(h(\partial U), \xi)=1$. By Lemma \ref{winding-intermediate}, there is $\zeta\in U$ with
$h(\zeta)=\xi$.  This gives a  Blaschke product $\widehat{B}= \widehat{B}_{\zeta, \mathbf x}\in \mathcal B_d$ with $d_{\mathbb D}(x_{q'}, \widehat{B}_{\zeta, \mathbf x}^l(c_{q}(\widehat{B}_{\zeta, \mathbf x})))=L$. Note that $x_{q'}=z_{q'}(\widehat{B}_{\zeta, \mathbf x})$, the proof is completed.
	\end{proof}

		\begin{proof}[Proof of Proposition \ref{regular-ns}]
			(1). Assume $1\notin {\rm supp}(S)$ and  let $q\in {\rm supp}(S)$ have multiplicity $\nu(q)\geq 2$.  
			 For any number $L\geq 0$, 	 choose two sequences $(b_n)_n, (c_n)_n$ both converging to $q$, and $d_{\mathbb D}(b_n, c_n)=L$ for all $n$. 
			 Let $(B_n)_n\subset \mathcal B_{d}$ be given by
			 $$Z(B_n)=Z(B)+ 1\cdot b_n+(\nu(q)-1)\cdot c_n, \ \forall n, $$ 
			By choosing a subsequence, we assume
			$f_{n}:=\Phi(B_n)\rightarrow f_L\in I_{\Phi}(D)$. Since $D$ is regular, $0$ is an attracting fixed point of $f_L$.
		
			By the choice of $B_{n}$, 
		 the preimage set	$f_{n}^{-1}(0)$ contains   $\phi_{f_{n}}(b_n)$  and $\phi_{f_{n}}(c_n)$ with hyperbolic distance 
			\begin{equation} \label{hyp-d-n}
				d_{U_{f_n}(0)}(\phi_{f_{n}}(b_n), \phi_{f_{n}}(c_n))=d_{\mathbb D}(b_n, c_n)=L, \ \forall n\geq 1.
				\end{equation}
			
			Passing to a subsequence, we assume $\phi_{f_{n}}(b_n)\rightarrow  b^L$,  $\phi_{f_{n}}(c_n)\rightarrow  c^L$. Then $b^L, c^L\subset f_L^{-1}(0)$.   Lemma \ref{kernel-att} and Remark \ref{kernel-att-pre} give the kernel convergence   
		$$(U_{f_{n}}(0), \phi_{f_{n}}(b_n))\rightarrow (U_{f_L}(b^L), b^L).$$
		
		By Lemma \ref{hyp-d-limit},   
		\begin{equation}\label{case1} c^L \in U_{f_L}(b^L) \  \text{ and } \ d_{U_{f_L}(b^L)}(b^L, c^L)=L.
			\end{equation}
		
		For each $L\geq 0$, by  Propositions \ref{div-eq} and \ref{unique-rm},   there is a unique conformal map $\phi_L: (\mathbb D, 0)\rightarrow (U_{f_L}(0), 0)$ satisfying that 
		$$B=\phi_{L}^{-1}\circ f_L \circ \phi_{L}, \  S=\phi_L^*(R_{f_L}^0)=\phi_L^*(Z_{f_L}^0).$$
	
		If $f_{L}=f_{L'}:=f$ for  $L,L'\geq 0$,  then $\phi_L, \phi_{L'}: (\mathbb D, 0)\rightarrow (U_{f}(0), 0)$ are conformal maps with 
		  $\phi_L(1)=\phi_{L'}(1)=\nu_f$ (by Proposition \ref{lp-0-ray}), hence 
		$\phi_L=\phi_{L'}:=\phi$. If follows that $f^{-1}(0)\cap L_{U_f(0), \phi(q)}=\{b^L,c^L\}=\{b^{L'}, c^{L'}\}$. By \eqref{case1}, $L=L'$. This means that different $L$ corresponds to different $f_L$.
		
		Note that $\{f_L; L\geq 0\}\subset I_{\Phi}(D)$. Therefore $I_{\Phi}(D)$  is not a singleton.
		
		
			
		

	(2). The idea is almost same as (1), but here we shall use Proposition \ref{prescribed-hyp-d}.  Suppose $S$ is simple and $1\notin {\rm supp}(S)$. Since $S$ has dynamical relation, there exist $q\in {\rm supp}(S)$ and a minimal integer $l\geq 1$  so that $q'=B^l(q)\in  {\rm supp}(S)\setminus\{q'\}$.  By Proposition \ref{prescribed-hyp-d}, for any $L\geq 0$ and any integer $n\geq 1$, there is  ${B}_n\in \mathcal B_d\cap    N_{1/n}(D)$ with the following property 
	$$d_{\mathbb D}(z_{q'}(B_n), B_n^l(c_{q}(B_n))=L.$$
	
	 By choosing a subsequence, we assume 
	$f_{n}:=\Phi(B_n)\rightarrow f_L\in I_{\Phi}(D)$. 
Note that	 $0$ is an attracting fixed point of $f_L$.  
 We further assume $\phi_{f_n}(z_{q'}(B_n))\rightarrow a$,   $\phi_{f_n}(B_n^l(c_{q}(B_n))\rightarrow b$ and $\phi_{f_n}(c_{q}(B_n))\rightarrow c$.  
  It follows that $f'_L(c)=0$, $f_L^l(c)=b$ and $f_L(a)=0$.  
	By Lemma \ref{kernel-att} and Remark \ref{kernel-att-pre},  we have the kernel convergence   
	$$(U_{f_{n}}(0), \phi_{f_n}(z_{q'}(B_n)))\rightarrow (U_{f_L}(a), a).$$
	
		  By Lemma \ref{hyp-d-limit},  
	 $b\in U_{f_L}(a)$ and $d_{U_{f_L}(a)}(a, b)=L$.
	
	 Note that $\{f_L; L\geq 0\}\subset I_{\Phi}(D)$. By  the same reasoning as (1),   different $L$ corresponds to different $f_L$, hence $I_{\Phi}(D)$  is not a singleton.
	 
	 
	 (3).  Since $1\in {\rm supp}(S)$, by Lemma \ref{degenerate-1} and Propositions \ref{div-eq} and \ref{unique-rm},  for any $\zeta \in \partial\mathbb D$, there exist  $f_\zeta\in I_{\Phi}(D)$, a conformal map $\phi_\zeta: (\mathbb D, 0)\rightarrow (U_{f_\zeta}(0), 0)$,  satisfying that 
	\begin{equation}\label{1-div-eq}\zeta B=\phi_{\zeta}^{-1}\circ f_\zeta \circ \phi_{\zeta}, \  S=\phi_\zeta^*(R_{f_\zeta}^0).  
		\end{equation}

If $f_{\zeta_1}=f_{\zeta_2}=f$ for  $\zeta_1, \zeta_2\in \partial\mathbb D$,  then the  conformal maps  $\phi_{\zeta_1}, \phi_{\zeta_2}: (\mathbb D,0)\rightarrow (U_f(0),0)$ have the same normalization $\phi_{\zeta_1}(1)=\phi_{\zeta_2}(1)$ (by Proposition \ref{lp-0-ray}). Hence $\phi_{\zeta_1}=\phi_{\zeta_2}$. This implies that $\zeta_1=\zeta_2$ by \eqref{1-div-eq}. 

This means that $I_{\Phi}(D)$ which contains $\{f_\zeta; \zeta\in \partial\mathbb D\}$ is not a singleton.
		\end{proof}

		\section{Singular divisors} \label{sing-div}

	In this section, we show

\begin{pro}\label{sing-imp3}   For any   $D=(B, S)\in \partial_{\rm sing}  {\mathcal B_d}$, we have 
	$$I_{\Phi}(D)\supseteq \{f_*\}, \ \text{ where } f_*(z)=z+z^d.$$
	The equality $I_{\Phi}(D)= \{f_*\}$ holds if and only if $S$ is simple and $1\notin {\rm supp}(S)$.
\end{pro}

	
	
	

Note that for any $D \in  \partial_{\rm sing}  {\mathcal B_d}$ and   $f\in I_{\Phi}(D)$,  $f$ has a fixed point at $0$.

	\begin{lem} \label{sing-multiplier}   Let $D=(B,S) \in  \partial_{\rm sing}  {\mathcal B_d}$.
		
		(1). If   $1\notin {\rm supp}(S)$, then for any  $f\in I_{\Phi}(D)$, we have $f'(0)=1$.
		
		(2).   If   $1\in {\rm supp}(S)$, then for any $\zeta\in \partial \mathbb D$, there is   $f\in I_{\Phi}(D)$ with $f'(0)=\zeta$.
		\end{lem}
	\begin{proof}
		(1).   Let $(B_n)_n$ be a sequence in $\mathcal B_{d}$ converging to $D$ algebraically, suppose $B_n$ has zeros $0, a_1(n), \cdots, a_{d-1}(n)$, then
		$$B_n'(0)=\prod_{k=1}^{d-1}A_k(n), \ \text{ where } A_k(n)=\frac{1-\overline{a_k(n)}}{1-a_k(n)}(-a_k(n)).$$
		Assume $\lim_{n}a_k(n)=q\in {\rm supp}(S)$.	The assumption $1\notin {\rm supp}(S)$  implies that $ A_k(n)\rightarrow 1$ and $B_n'(0)\rightarrow 1$.   It follows that for any  $f\in I_{\Phi}(D)$,   $f'(0)=1$.

		(2).  If   $1\in {\rm supp}(S)$,  by Lemma \ref{degenerate-1},  for any $\zeta\in \partial \mathbb D$, there is a sequence $(B_n)_n$  in $\mathcal B_{d}$  such that 
		$B_n\rightarrow D$ algebraically, and $B_{n}$ converges  to $\zeta B$ in 	$\C\setminus {\rm  supp}(S)$.  
	A subsequence of  $f_n=\Phi(B_n)$ has a limit $f\in I_{\Phi}(D)$ with $f'(0)=\zeta$.
		\end{proof}

	\begin{lem} \label{para-m}  Let $f\in   \partial_{\rm sing}  {\mathcal H_d}$ have a parabolic fixed point at $0$ with $f'(0)=1$ and parabolic  multiplicity  \footnote{The {\it parabolic  multiplicity}   is the minimal integer $m$ so that $f(z)=z(1+az^m+o(z^m))$ near $0$,  where $a\neq 0$, see \cite{BE}. It equals the number of the immediate parabolic basins of $0$.}  $m\geq 1$. 
	Then $K(f)\setminus\{0\}$ has exactly $m$ connected components.
\end{lem}

	\begin{proof}
	We label the immediate parabolic basins of $f$ at $0$ by 
	$A_1, \cdots, A_m$.  	To prove the lemma, it suffices to show $\bigcap_{1\leq k\leq m} L_{A_k, 0}=\{0\}$.
	
		\begin{figure} \label{fig-para-limb}
		\begin{center}
			\includegraphics[height=4cm]{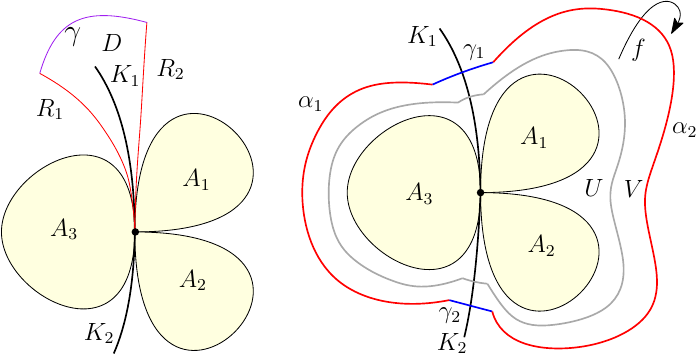}
		\end{center}
		\caption{The parabolic basins and limbs ($l=2$).}
		\label{fig: fig-para-limb}
	\end{figure}

	Note that the $f$-parabolic point  $0$ is the landing point of finitely many external rays \cite[\S 18]{Mil06}, and the number of these rays equals the number of the connected components of $K(f)\setminus \{0\}$ \cite[Corollary 6.7]{Mc94}. 
	If the conclusion is false, then $ \bigcap_{1\leq k\leq m} L_{A_k, 0}$ consists of finitely many connected  components $K_1, \cdots, K_l$, $l\geq 1$. We claim that
	each $K_s$ contains at least one critical point of $f$.  To see this, observe that $K_s$ is in a domain $D$   bounded by two invariant external rays $R_1, R_2$  and an  equipotential  segment $\gamma$, see Figure \ref{fig: fig-para-limb} (left).    If $K_s$ contains no critical point of $f$, then for each $j$,  there is a unique  connected component $D_j$ of $f^{-j}(D)$ whose boundary $\partial D_j$ contains segments in  $R_1\cup R_2$.  Note that $K_s\subset D_j$.  
  The Shrinking Lemma \cite[\S 12]{LM97} yields ${\rm diam}(\partial D_j)\rightarrow 0$ as $j\rightarrow \infty$. However, this contradicts the fact that ${\rm diam}(\partial D_j)\geq {\rm diam}(K_s)>0$ for all $j$.  Hence each $K_s$ contains at least one critical point of $f$.

	Note that each $K_s$ intersects a repelling petal of $0$, which enables the construction of a polynomial-like mapping as follows:
	
	Let $V$ be bounded by $l$ equipotential segments $\alpha_1, \cdots, \alpha_l$  in the basin of $\infty$ of $f$, together with $l$ arcs $\gamma_1,\cdots, \gamma_l$ in the  $l$ repelling petals, see Figure \ref{fig: fig-para-limb} (right). We require that   each $\gamma_s$ takes the form  $\{{\rm Re}(w)=L_s\}$ in the   corresponding  repelling Fatou coordinate.  By shrinking $V$, we may assume $\overline{V}\cap (K_1\cup\cdots \cup K_l)$ contains no critical values. 
	Let $U$ be the component of $f^{-1}(V)$ containing $\bigcup_{1\leq k\leq m}\overline{A_k}$. Then $f: U\rightarrow V$ is a polynomial-like mapping with degree ${\rm deg}(f|_U)=d-\sum_{s=1}^l d_s<d$, where $d_s$ is the number of $f$-critical points in $K_s$.   Its filled Julia set $K(f|_U):=\bigcap_{j\geq 0} f^{-j}(V)$ is connected and contains $\bigcup_{1\leq k\leq m}\overline{A_k}$.
	
	Now for any sequence $(f_n)_n$ in $\mathcal H_d$ approaching $f$ and for  large $n$, $f|_U$ induces  a polynomial-like restriction  $f_n:U_n\rightarrow V$ of $f_n$  with   degree ${\rm deg}(f|_U)$ and   $0\in U_n$.  Since  the filled Julia set $K(f_n|_{U_n})=\bigcap_{j\geq 0} f_n^{-j}(V)$ contains $U_{f_{n}}(0)$,  the degree ${\rm deg}(f_n|_{U_n})\geq d$. Contradiction!
	\end{proof}

	\begin{lem} \label{to-zero}  Let $f\in   \partial_{\rm sing}  {\mathcal H_d}$
		have a parabolic fixed point at $0$ with $f'(0)=1$.
		Then for any sequence $(f_n)_n$ in $\mathcal H_d$ converging to $f$, the   conformal maps $\phi_{f_n}:\mathbb D\rightarrow U_{f_n}(0)$ converge (locally and uniformly in $\mathbb D$) to $0$.	 
\end{lem}
\begin{proof} By Koebe distortion theorem \cite[Theorem 5.3]{A},
	$$|\phi_{f_n}(z)|\leq |\phi_{f_n}'(0)|\frac{|z|}{(1-|z|)^2}, \ |\phi_{f_n}'(0)|\leq 4 \cdot {\rm dist}(0, J(f_n)),$$
	where ${\rm dist}(0, J(f_n))=\min_{z\in J(f_n)}|z|$.
	Since $J(f_n)$ contains an $f_n$-repelling fixed point  which tends to $0$ as $n\rightarrow 0$, we get $\phi_{f_n}'(0)\rightarrow 0$. 
	The 
	 convergence $\phi_{f_n}\rightarrow 0$ follows immediately.
	\end{proof}

	\begin{lem} \label{to-zero2}  Let $f\in   \partial_{\rm sing}  {\mathcal H_d}$ have a parabolic fixed point at $0$ with $f'(0)=1$.
	 Let $(f_n)_n$ be a sequence  in $\mathcal H_d$ converging to $f$. 
		There exist a full measure subset $E$ of $\partial\mathbb D$, and a subsequence $(f_{n_k})_k$  of $(f_n)_n$ such that for any  $\zeta\in E$, the Euclidean length of the curve $\phi_{f_{n_k}}([0,\zeta])$
	  converges to $0$ as $k\rightarrow 0$.	 
\end{lem}
\begin{proof}  By  Proposition \ref{convergent-rays} and Lemma \ref{to-zero}, the length function 
	$L_n: \partial \mathbb D\rightarrow [0,+\infty]$,  defined by $L_n(\xi)=\int_{0}^1|\phi_{f_n}'(r\xi)|dr$, converges to $0$ in the $L^1$-norm as $n\rightarrow \infty$. Hence there is a full measure subset $E$ of $\partial\mathbb D$ and subsequence 
	$(L_{n_k})_k$ so that $L_{n_k}$ converges to $0$ pointwisely in $E$. 
		\end{proof}

 \vspace{5pt}
\noindent \textbf{Proof of Proposition \ref{sing-imp3}.}  Proposition \ref{sing-imp3} follows from   Lemmas \ref{sim-imp} and \ref{non-sim}.    Corollary 1.1 is an immediate consequence of  Proposition \ref{sing-imp3}.
 	\begin{lem} \label{sim-imp}   Let $D=(B,S) \in  \partial_{\rm sing}  {\mathcal B_d}$ with $S$ simple and $1\notin {\rm supp}(S)$, then
			$$I_{\Phi}(D)= \{f_*\} \ \text{ with } f_*(z)=z+z^d.$$
\end{lem}
\begin{proof} Let $f\in I_{\Phi}(D)$.  By Lemma \ref{sing-multiplier}, $f$ has a parabolic fixed point at $0$ with  $f'(0)=1$.  To show $f=f_*$, it suffices to show that the parabolic multiplicity $m$ of $f$ at $0$ equals $d-1$.
	
	Assume by contradiction that $m<d-1$.  By Lemma \ref{para-m},   $K(f)\setminus\{0\}$ has exactly $m$ connected components. Since $f$ has $d-1$ critical points in $K(f)$,   one component of $K(f)\setminus\{0\}$, denoted as $L$,    contains at least two critical points $c_1(f), c_2(f)$ (possibly same). 
 Take a sequence $(B_n)_n$ in $\mathcal B_d$ so that $B_n\rightarrow D$  algebraically, and $f_n:=\Phi(B_n)\rightarrow f$. There are  critical points $c_1(f_n)$ and $c_2(f_n)$ of $f_n$ so that $c_1(f_n)\rightarrow c_1(f)$ and  $c_2(f_n)\rightarrow c_2(f)$.
 
 Note that the fixed point $0$ of $f$ splits into an attracting fixed point $0$ and $m$-repelling fixed points $r_1(f_n), \cdots, r_m(f_n)$ of $f_n$. Further 
 $$\delta_n:=\max_{1\leq k\leq m} |r_k(f_n)|\rightarrow 0 \ \text{ as } \ n\rightarrow \infty.$$

 Take a positive  $\epsilon_0<\min\{|c_1(f)|, |c_2(f)|\}$.  It is clear that for large $n$, 
 $$\delta_n<\epsilon_0<\min\{|c_1(f_n)|, |c_2(f_n)|\}.$$
  Let $W_n^0$ be the connected component of $\mathbb D(0, \epsilon_0)\cap U_{f_n}(0)$ containing $0$.
 The fact   $c_1(f), c_2(f)\in L$ implies that $c_1(f_n)$ and $c_2(f_n)$ are in the same connected component $W_n$ of $U_{f_n}(0)\setminus \overline{W_n^0}$ \footnote{The asymptotic shapes of the filled Julia set $K(f_n)$ are implicitly studied by Oudkerk \cite{O} using gate structure and parabolic implosion. }.   See Figure \ref{fig: para-singular}.

Since $W_n$ is path-connected, there is a curve $\gamma_n$ in $W_n$ connecting $c_1(f_n)$ and $c_2(f_n)$.
It follows that $\phi_{f_n}^{-1}(\gamma_n)$ is a curve in $\mathbb D\setminus\{0\}$ connecting the two critical points $c_1^{(n)}=\phi_{f_n}^{-1}(c_1(f_n)), c_2^{(n)}=\phi_{f_n}^{-1}(c_2(f_n))$ of  $B_n$.   By Theorem \ref{bet} and passing to a subsequence if necessary, we assume $c_1^{(n)}\rightarrow q_1$ and $c_2^{(n)}\rightarrow q_2$, where $q_1, q_2\in {\rm supp}(S)$.  Since $S$ is simple, we have $q_1\neq q_2$. 

By Lemma \ref{to-zero2},  there  exist a subsequence $(f_{n_k})_k$ and  $\zeta\in \partial \mathbb D\setminus\{q_1, q_2\}$ so that $[0, \zeta]\cap \phi_{f_{n_k}}^{-1}(\gamma_{n_k})\neq \emptyset$ and the Euclidean length of $\phi_{f_{n_k}}([0, \zeta])$ tends to zero as $k\rightarrow \infty$. 

 \begin{figure} 
	\begin{center}
		\includegraphics[height=4cm]{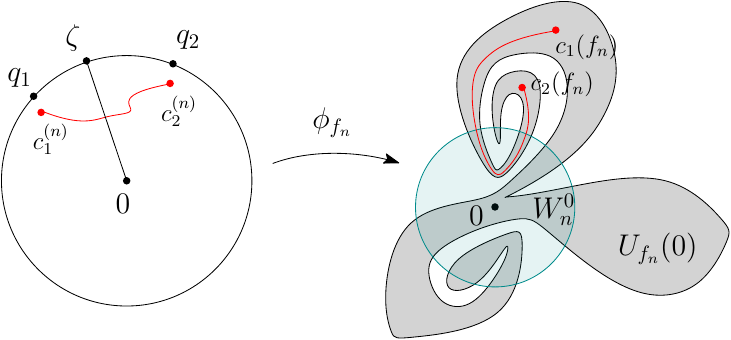}
	\end{center}
	\caption{The perturbation $f_n$ of $f$.}
	\label{fig: para-singular}
\end{figure}

On the other hand, the fact $\phi_{f_{n_k}}([0, \zeta])\cap \gamma_{n_k}\neq \emptyset$ means that $\phi_{f_{n_k}}([0, \zeta])$ is a  curve in $U_{f_{n_k}}(0)$ connects $0$ and a point on $\gamma_{n_k}$.  Since $\gamma_{n_k}\cap W_{n_k}^0=\emptyset$,   the length of  $\phi_{f_{n_k}}([0, \zeta])$ 
   is at least $\epsilon_0$. This is  a contradiction. 
	\end{proof}

	\begin{lem} \label{non-sim}   Let $D=(B,S) \in  \partial_{\rm sing}  {\mathcal B_d}$. Assume $S$ is not simple  or $1\in {\rm supp}(S)$, then
	$$I_{\Phi}(D)\supsetneq \{f_*\} \ \text{ with } f_*(z)=z+z^d.$$
\end{lem}
\begin{proof} By the density of simple divisors,  there is  a sequence of simple $(S_n)_n$ in ${\rm Div}_{d-1}(\partial \mathbb D)$ so that $S_n\rightarrow S$ and $1\notin  {\rm supp}(S_n)$.  
For any $\epsilon>0$, there exist an integer $n>0$ and a  number $r\in(0,\epsilon)$ with $N_{r}(D_n)\subset N_{\epsilon}(D)$,  where $D_n=(B, S_n)$. It follows that  $I_{\Phi}(D_n)\subset \overline{\Phi(N_{r}(D_n))}\subset \overline{\Phi(N_{\epsilon}(D))}$.   By Lemma \ref{sim-imp},  $I_{\Phi}(D_n)= \{f_*\}$. Hence $f_*\in \overline{\Phi(N_{\epsilon}(D))}$ for all $\epsilon>0$, this  implies  that $f_*\in \bigcap_{\epsilon>0}\overline{\Phi(N_{\epsilon}(D))}=I_{\Phi}(D)$. 
	 
	 In the following, we show $I_{\Phi}(D)\neq \{f_*\}$.  If $1\in {\rm supp}(S)$, it follows from Lemma \ref{sing-multiplier}  (2). 
	 
	  Assume $S$ is not simple.
	 Write $S=\sum_{q\in {\rm supp}(S)} \nu(q)\cdot q$ and let
	 $$R_n= \sum_{q\in {\rm supp}(S)} \nu(q)\cdot (1-1/n)q \in {\rm Div}_{d-1}(\mathbb D), \ n\geq 1.$$
	 Clearly $R_n\rightarrow S$ in ${\rm Div}_{d-1}(\overline{\mathbb D})$.
	 Let  $B_n\in \mathcal B_{d}$ have free ramification divisor  $R_n$.
	 By Theorem \ref{bet}, $B_n\rightarrow D$ algebraically.
	  Let $f$  be  one limit map of the sequence  $f_n=\Phi(B_n)$, then $f\in I_{\Phi}(D)$.  Since $S$ is not simple,  $f_n$ has non-simple critical points for all $n$. It follows that $f$ has  non-simple critical points.  Since all critical points of $f_*$ are simple, we have $f\neq f_*$.  
	 \end{proof}
 


	\section{Proof of the main theorems} \label{proof-main}
	
	In this section, we shall prove Theorems \ref{bet-main},  \ref{lc-R} and  \ref{cusp-dense}.  For each $f\in \mathcal P_d$ and $\epsilon>0$,  let $\mathcal N_\epsilon(f)$ denote  the $\epsilon$-neighborhood of $f$ in $\mathcal P_d$.
	
	\begin{pro} \label{decomposition}  For  any $f\in \partial_{\rm reg}\mathcal H_d$, there is a unique  divisor  $D=(B,S) \in  \partial_{\rm reg}  \mathcal B_{d}$ so that $f\in I_{\Phi}(D)$.  This implies the following decomposition 
		$$ \partial_{\rm reg}\mathcal H_d=\bigsqcup_{D\in  \partial_{\rm reg}  \mathcal B_{d}} I_{\Phi}(D).$$
	\end{pro}
\begin{proof}Let $(f_n)_n$ be  a sequence in $\mathcal H_d$ converging to $f$.  Since $\overline{\mathcal B_d}$ is compact, the   sequence
	$(\Psi(f_n))_n$ has an accumulation divisor $D\in  \partial \mathcal B_{d}$. Hence $f\in I_{\Phi}(D)$. The assumption $f\in \partial_{\rm reg}\mathcal H_d$ implies that $D\in \partial_{\rm reg}  \mathcal B_{d}$.
	
	Suppose that there are $D_1=(B_1, S_1), D_2=(B_2, S_2)\in \partial_{\rm reg} \mathcal B_d$ so that $f\in I_{\Phi}(D_1)\cap I_{\Phi}(D_2)$.  By Proposition   \ref{unique-rm},  there is a unique conformal map $\phi_f:(\mathbb D, 0)\rightarrow (U_f(0), 0)$  so that
	for any sequence $(f_n)_n$ in $\mathcal H_d$ converging to $f$, the conformal maps 
	$(\phi_{f_n})_n$ converge to $\phi_f$ in $\mathbb D$.  
	We may choose two sequences  $(f_n)_n$ and $(g_n)_n$ converging to $f$ so that $\Psi(f_n)\rightarrow D_1$ and $\Psi(g_n)\rightarrow D_2$ .
By Proposition \ref{div-eq}, 
		$$\zeta_k B_k=\phi_{f}^{-1}\circ f\circ \phi_{f}, \  S_k=\phi_f^*(R_f^0), \ \ k=1,2,$$
	for some $\zeta_k\in \partial \mathbb D$.  It follows that $\zeta_1 B_1=\zeta_2 B_2$ and  $S_1=S_2$. The former equality   implies
	 $Z(B_1)= Z(B_2)$. Since $B_1(1)=B_2(1)=1$, we have $B_1=B_2$. 
			\end{proof}
		
  


By Proposition \ref{decomposition}, there is a well-defined map   
	$$\Pi:  
\begin{cases}\partial_{\rm reg}\mathcal H_d\rightarrow  \partial_{\rm reg}  \mathcal B_{d}, \\
	f\mapsto D_f
\end{cases}$$
where $D_f$ is the unique divisor in $\partial_{\rm reg}  \mathcal B_{d}$ so that $f\in I_{\Phi}(D_f)$.

\begin{cor}\label{unique-zeta} For any $f\in \partial_{\rm reg}\mathcal H_d$, write $D=(B,S)=\Pi(f)$. There is a unique number $\zeta=\zeta(f)\in \partial \mathbb D$ such that for any sequence  $(f_n)_n$ in $\mathcal H_d$ converging to $f$, the Blaschke products $B_n=\Psi(f_n)$ converge to $\zeta B$ in $\C-{\rm supp}(S)$. 
 \end{cor}
\begin{proof} By Proposition \ref{decomposition}, $B_n\rightarrow D$ algebraically. By  Lemma \ref{degenerate-0}, if $1\notin {\rm supp}(S)$, then $B_n$ converges to $B$ in $\C-{\rm supp}(S)$, in this case $\zeta=1$; if $1\in {\rm supp}(S)$, then there exist a subsequence $(B_{n_k})_k$ and a number $\zeta\in \partial \mathbb D$ so that 
	$B_{n_k}$  converges to $\zeta B$ in $\C-{\rm supp}(S)$.  For the latter,   Proposition \ref{unique-rm} gives a unique conformal map $\phi_f:(\mathbb D, 0)\rightarrow (U_f(0), 0)$   so that 
	$\phi_{f_n}$ converges to $\phi_f$ in $\mathbb D$. 
	By  Proposition \ref{div-eq},  $\zeta B=\phi_{f}^{-1}\circ f\circ \phi_{f}$. 
	This equality implies that  $\zeta$ is independent of   the subsequence. Therefore the whole sequence $(B_n)_n$ converges to $\zeta B$ in $\C-{\rm supp}(S)$.
	\end{proof}


	\begin{pro} \label{proj-cont}    The map $\Pi:  \partial_{\rm reg}\mathcal H_d\rightarrow  \partial_{\rm reg}  \mathcal B_{d}$ is continuous.
\end{pro}
\begin{proof} Let $(f_n)_n$ be a sequence in $\partial_{\rm reg}\mathcal H_d$ converging to $f\in \partial_{\rm reg}\mathcal H_d$. Since $\partial \mathcal B_d$ is compact, passing to a subsequence,  we assume $(\Pi(f_n))_n$ has a limit $D\in \partial \mathcal B_d$. Since $f\in \partial_{\rm reg}\mathcal H_d$, we have $D\in \partial_{\rm reg}  \mathcal B_{d}$.
	
	 In the following, we show $f\in I_{\Phi}(D)$. For each $n\geq 1$, choose $g_n\in \mathcal N_{1/n}(f_n)\cap \mathcal H_d$ so that $\Psi(g_n)\in N_{1/n}(\Pi(f_n))$, then 
	$$g_n\rightarrow f, \ \Psi(g_n)\rightarrow D.$$
	Hence $f\in I_{\Phi}(D)$,   equivalently $D=\Pi(f)$, establishing the continuity of $\Pi$.
	\end{proof}

	\begin{rmk} \label{ext-psi}    The homeomorphism  $\Psi:   \mathcal H_d\rightarrow   \mathcal B_{d}$  extends to a continuous map $\overline{\Psi}:   \mathcal H_d\sqcup \partial_{\rm reg}\mathcal H_d\rightarrow   \mathcal B_{d}\sqcup \partial_{\rm reg}  \mathcal B_{d}$. 
\end{rmk}
\begin{proof} Set $\overline{\Psi}|_{ \mathcal H_d}=\Psi$ and
	 $\overline{\Psi}|_{\partial_{\rm reg}\mathcal H_d}=\Pi$.
	\end{proof}
 
 \begin{proof}[Proof of Theorem \ref{bet-main}]
 	By Propositions \ref{imp-singleton}, \ref{regular-ns} and \ref{sing-imp3}, we get the necessary and sufficient conditions for  $D\in \partial \mathcal B_d$ which allows  $\Phi$-extension.  The continuous extension  
 	$\overline{\Phi}:\mathcal B_d\sqcup \mathcal R\sqcup \mathcal S\rightarrow \overline{\mathcal H_d}$
   is defined as follows: if  $D\in \mathcal R$, then  $\overline{\Phi}(D)=f$, where  $f$ is the unique map in $I_{\Phi}(D)$ (by Proposition \ref{imp-singleton});  if  $D\in \mathcal S$, then $\overline{\Phi}(D)=f_*$
(by Proposition \ref{sing-imp3}).

 	It remains to show that $\overline{\Phi}|_{\mathcal R}: \mathcal R\rightarrow \overline{\Phi}(\mathcal R)$ is a homeomorphism.
 	
 	First, the equality  $\Pi\circ \overline{\Phi}|_{\mathcal R}={\rm id}$ implies that $\overline{\Phi}|_{\mathcal R}$ is a bijection. For any $D\in \mathcal R$, by the proven fact $I_{\Phi}(D)=\bigcap_{\delta>0} \overline{\Phi(N_{\delta}(D))}=\{f\}$ in Proposition \ref{imp-singleton}, we conclude that for any 
 	$\epsilon>0$, there is a $\delta>0$ so that ${\rm diam}(\overline{\Phi(N_{\delta}(D))})<\epsilon$.
 	For any $E\in U_{\delta}(D)\cap \mathcal R$,  choose  small $\delta_E>0$ so that 
 	$N_{\delta_E}(E)\subset N_{\delta}(D)$, it follows that $\{\overline{\Phi}(E)\}=I_{\Phi}(E)\subset \overline{\Phi(N_{\delta_E}(E))}\subset \overline{\Phi(N_{\delta}(D))}$ . Hence $\overline{\Phi}(E)$ is in the $\epsilon$-neighborhood of $f$. This shows the continuity of $\overline{\Phi}|_{\mathcal R}$.
 	
 	By Proposition \ref{proj-cont},  $\overline{\Phi}|_{\mathcal R}^{-1}=\Pi$ is continuous. Hence 
 	 $\overline{\Phi}|_{\mathcal R}: \mathcal R\rightarrow \overline{\Phi}(\mathcal R)$ is a homeomorphism.
 		\end{proof}
 	
 \begin{proof}[Proof of Theorem \ref{lc-R}]   For any $f\in \overline{\Phi}(\mathcal R)$,   let $D=\Pi(f) \in \partial_{\rm reg} \mathcal B_{d}$.  Since $\partial_{\rm reg} \mathcal B_{d}$ is open  in $\partial \mathcal B_{d}$, there is  $\epsilon_D>0$ so that $U_{\epsilon_D}(D)\cap  \partial \mathcal B_{d}\subset  \partial_{\rm reg} \mathcal B_{d}$.  
 	For each $0<\epsilon\leq  \epsilon_D$, note that  $U_{\epsilon }(D)\cap  \partial \mathcal B_{d}$ is an open and path-connected subset of $ \partial \mathcal B_{d}$ containing $D$.
 By Proposition \ref{proj-cont}, 
 $$\mathcal U_\epsilon:=\Pi^{-1}(U_{\epsilon }(D)\cap  \partial \mathcal B_{d})$$
 is an  open  subset of $\partial_{\rm reg} \mathcal H_d$ containing  $f$. Since $\partial_{\rm reg} \mathcal H_d$ is an open subset of $\partial \mathcal H_d$,  $\mathcal U_\epsilon$ is also open in $\partial \mathcal H_d$.
 
 In the following, we show that  $\{\mathcal U_\epsilon; 0<\epsilon\leq  \epsilon_D\}$ is a family of   connected neighborhoods of $f$ with ${\rm diam}(\mathcal U_\epsilon)\rightarrow 0$ as $\epsilon\rightarrow 0$.

 For each $E\in U_{\epsilon }(D)\cap  \partial \mathcal B_{d}$, let $\gamma_E: [0,1]\rightarrow U_{\epsilon }(D)\cap  \partial \mathcal B_{d}$ be a path with $\gamma_E(0)=D$ and $\gamma_E(1)=E$. 
 Let $r_0>0$ be small so that 
$\bigcup_{E'\in \gamma_E}U_{r_0}(E')\subset U_{\epsilon }(D)$. 
We first show that 
 \begin{equation}\label{eq-set}
 \Pi^{-1}(\gamma_E)=K(E): =\bigcap_{0<r<r_0}\overline{\Phi\Big(\bigcup_{E'\in \gamma_E}N_r(E')\Big)}.
 \end{equation}

To see this, first note that $ \Pi^{-1}(\gamma_E)\subset K(E)$.  Conversely,  for any $g\in
K(E)$, there is a sequence $(g_n)_n$ in   $\mathcal H_d$, and a sequence $(E_n)_n$ in $\gamma_E$ so that 
$g_n\rightarrow g$ and $\Psi(g_n)\in N_{1/n}(E_n)$ for all $n\geq 1$. Passing to a subsequence, and by the compactness of $\gamma_E$,  we assume $E_n\rightarrow E_* \in \gamma_E$.
 It follows that $g\in I_{\Phi}(E_*)\subset\Pi^{-1}(\gamma_E)$. This establishes the equality \eqref{eq-set}.
 
By \eqref{eq-set}, $\Pi^{-1}(\gamma_E)$ is connected, because it is a shrinking sequence of connected compacta.
The connectivity of $\mathcal U_\epsilon$  follows from the facts:
$$\mathcal U_\epsilon=\bigcup_{E\in U_{\epsilon }(D)\cap  \partial \mathcal B_{d}}\Pi^{-1}(\gamma_E), \ \ f\in \bigcap_{E\in U_{\epsilon }(D)\cap  \partial \mathcal B_{d}}\Pi^{-1}(\gamma_E).$$

It remains to show ${\rm diam}(\mathcal U_\epsilon)\rightarrow 0$ as $\epsilon\rightarrow 0$.
We claim $\mathcal U_\epsilon\subset \overline{\Phi(N_{\epsilon}(D))}$.  In fact, for any $E\in U_{\epsilon }(D)\cap  \partial \mathcal B_{d}$, there is $r>0$ so that $N_r(E)\subset N_{\epsilon }(D)$. It follows that 
$\Pi^{-1}(E)=I_{\Phi}(E)\subset \overline{\Phi(N_{r}(E))}\subset \overline{\Phi(N_{\epsilon}(D))}$. The claim follows. 

By the claim and Proposition \ref{imp-singleton}, ${\rm diam}(\mathcal U_\epsilon)\leq {\rm diam}(\overline{\Phi(N_{\epsilon}(D))})\rightarrow 0$ as $\epsilon\rightarrow 0$, completing the proof.
 	\end{proof}
 	

 \begin{proof}[Proof of Theorem \ref{cusp-dense}] For any $f\in \overline{\Phi}(\mathcal R)$, let $D=(B_0,S_0)=\Pi(f) \in \mathcal R$.  By  Proposition \ref{imp-singleton}, $ {\rm diam}(\overline{\Phi(N_{\delta}(D))})\rightarrow 0$ as $\delta\rightarrow 0$.  For any given $\epsilon>0$, choose $\delta>0$ so that 
 	${\rm diam}(\overline{\Phi(N_{\delta}(D))})<\epsilon$.
 	Since $S_0$ is simple, write $S_0=\sum_{k=1}^{{\rm deg}(S_0)}1\cdot q_k$, where  ${\rm deg}(S_0)=d-{\rm deg}(f|_{U_f(0)})$.  We further assume $\delta$ is small so that $\overline{\mathbb D(q_k, \delta)}, 1\leq k\leq {\rm deg}(S_0)$ are pairwise disjoint,  without containing $1$ (since $1\notin {\rm supp}(S_0)$).
 	
 	
 	

 		Let $l={\rm deg}(S_0)-m-n$. For each $d-l<k\leq d$,
 	 choose $q'_k\in \mathbb D\cap \mathbb D(q_k, \delta)$. Let $B\in \mathcal B_{d-m-n}$ be the Blaschke product with  zero divisor 
 	 $Z(B_0)+\sum_{d-l<k\leq d}1\cdot q'_k$.  
 	 There are  $B$-periodic points $a_1, \cdots, a_{m}, b_1,\cdots, b_n \in \partial \mathbb D$, coming from $(m+n)$  different $B$-periodic cycles. 
 	 Since $B$-periodic points are dense in $\partial \mathbb D$, we may assume $a_k\in \partial \mathbb D\cap \mathbb D(q_k, \delta)$, $1\leq k\leq m$. For $1\leq j\leq n$, since $\bigcup_{l\geq 0} B^{-l}(b_j)$ is also dense in $\partial \mathbb D$, there is a $B$-aperiodic point $b'_j\in \mathbb D(q_{m+j}, \delta)\cap \bigcup_{l\geq 0} B^{-l}(b_j)$. Now set
 	 $$S=\sum_{k=1}^{m}1\cdot a_k+\sum_{j=1}^{n}1\cdot b'_j, \ \ E=(B, S).$$
 	 Clearly $E\in U_{\delta}(D)$,  $S$ is simple and $1\notin {\rm supp}(S)$.  By the choices of $a_k$ and $b'_j$, $E$ has no dynamical relation. Hence $E\in \mathcal R$.  
 	 
 	 
 	  By Proposition \ref{imp-singleton}, there is a unique map $g$ in $I_{\Phi}(E)\subset \overline{\Phi(N_{\delta}(D))}\subset  \mathcal N_\epsilon(f)$.  By Lemma \ref{char-map-imp}, this $g$ is geometrically finite,  having $m$ parabolic cycles and  $n$ critical points on $\partial U_g(0)$.
 	  	\end{proof}
   	
   		\begin{rmk}  
   			 $\overline{\Phi}(\mathcal R)$ is not dense in $\mathcal \partial \mathcal H_d$ for $d\geq 4$. 
   				\end{rmk}
   			\begin{proof}[Sketch of proof]
   				For any $d\geq 4$, there is  a divisor $D=(B, 2\cdot q)\in \partial_{\rm reg}  \mathcal B_{d}$ with  $q\in \partial \mathbb D \setminus\{1\}$.  By Proposition \ref{regular-ns}, $I_{\Phi}(D)$ is  not a singleton. 
   			Take  $f\in I_{\Phi}(D)$ so that a  component $V_f$ of $f^{-1}(U_f(0))\setminus U_f(0)$  contains a critical point $c_1(f)$, and $\partial V_f\cap \partial U_f(0)$ consists of another critical point $c_2(f)$. 
   			Then there is a neighborhood of $\mathcal N_\epsilon(f)$ of $f$ so that 
   			$ \mathcal N_\epsilon(f)\cap \partial\mathcal H_d\subset \partial_{\rm reg} \mathcal H_d$,  and $\Pi(g)$ takes the form $(B_g, 2 \cdot q_g)$ for any $g\in \mathcal N_\epsilon(f)\cap \partial\mathcal H_d$.
   			Hence $\mathcal N_\epsilon(f)\cap \overline{\Phi}(\mathcal R)=\emptyset$.
   			\end{proof}

	\bibliographystyle{alpha} 

\end{document}